\documentclass[12pt]{article}
\usepackage{amssymb,amsmath,amsthm,bbm,mathrsfs}
\usepackage[margin=1in]{geometry}
\usepackage{enumitem}
\usepackage{graphicx,color}
\usepackage[small]{caption}
\usepackage[bookmarksdepth=2]{hyperref}

\definecolor{bda}{rgb}{0.1,0.5,0.1}
\definecolor{bdb}{rgb}{0.1,0.8,0.1}
\newcommand{\mk}[1]{#1}

\newcommand{\MK}[1]{#1}

\theoremstyle{plain}
\newtheorem{theorem}{Theorem}
\newtheorem*{theorem*}{Theorem}
\newtheorem{lemma}{Lemma}

\newtheorem{proposition}{Proposition}

\theoremstyle{definition}

\theoremstyle{remark}

\newcommand{\A}{\mathscr{A}}
\newcommand{\B}{\mathscr{B}}
\newcommand{\lambdau}{\smash{\overline{\lambda}}}
\newcommand{\lambdal}{\smash{\underline{\lambda}}}

\newcommand{\dom}{\mathscr{D}}
\newcommand{\form}{\mathscr{E}}

\newcommand{\R}{\mathbf{R}}

\newcommand{\sub}{\subseteq}
\newcommand{\ph}{\varphi}

\newcommand{\eps}{\varepsilon}

\newcommand{\set}[1]{\left\{ #1 \right\}}

\newcommand{\expr}[1]{\left( #1 \right)}

\newcommand{\ignore}[1]{}

\newcommand{\rad}{{r}}
\newcommand{\BB}{{\mathcal{B}}}
\newcommand{\KK}{{\mathcal{K}}}
\newcommand{\LL}{{\mathcal{L}}}
\newcommand{\MM}{{\mathcal{M}}}
\newcommand{\PP}{{\mathcal{P}}}
\newcommand{\QQ}{{\mathcal{Q}}}
\newcommand{\RR}{{\Lambda}}

\usepackage{ifthen}
\newcommand{\formula}[2][nolabel]%
{%
 \ifthenelse{\equal{#1}{nolabel}}%
 {\begin{align*} #2 \end{align*}}%
 {%
  \ifthenelse{\equal{#1}{}}%
  {\begin{align} #2 \end{align}}%
  {\begin{align} \label{#1} \begin{aligned} #2 \end{aligned} \end{align}}%
 }%
}

\newmuskip\pFqskip
\mathchardef\pFcomma=\mathcode`,

\begin{document}

%
%                            ---------- o ----------
%

\title{Eigenvalues of the fractional Laplace operator \\ in the unit ball}
\author{
{Bart{\l}omiej Dyda\footnotemark[1] \footnotemark[3]} , \;  
{Alexey Kuznetsov\footnotemark[2] \footnotemark[4]} , \;
{Mateusz Kwa{\'s}nicki\footnotemark[1] \footnotemark[5]}
}
\date{\today}

\maketitle
{
\renewcommand{\thefootnote}{\fnsymbol{footnote}}
\footnotetext[1]{Faculty of Pure and Applied Mathematics, Wroc{\l}aw University of \mk{Science and} Technology, ul. Wybrze{\.z}e Wyspia{\'n}skiego 27, 50-370 Wroc{\l}aw, Poland. Email: \{bartlomiej.dyda,mateusz.kwasnicki\}@pwr.edu.pl}
\footnotetext[2]{Dept. of Mathematics and Statistics,  York University,
4700 Keele Street, Toronto, ON, M3J 1P3, Canada.   Email: kuznetsov@mathstat.yorku.ca}
\footnotetext[3]{Supported by Polish National Science Centre (NCN) grant no. 2012/07/B/ST1/03356}
\footnotetext[4]{Research supported by the Natural Sciences and Engineering Research Council of Canada}
\footnotetext[5]{Supported by Polish National Science Centre (NCN) grant no. 2011/03/D/ST1/00311}
}

\begin{abstract}
We describe a highly efficient numerical scheme for finding two-sided bounds for the eigenvalues of the fractional Laplace operator
$(-\Delta)^{\alpha/2}$ in the unit ball $D \subset \R^d$, with a Dirichlet condition in the complement of $D$. The standard Rayleigh--Ritz variational method is used for the upper bounds, while the lower bounds involve the less-known Aronszajn method of intermediate problems. Both require explicit expressions for the fractional Laplace operator applied to a linearly dense set of functions in $L^2(D)$. We use appropriate Jacobi-type orthogonal polynomials, which were studied in a companion paper~\cite{bib:dkk15}. Our numerical scheme can be applied analytically when polynomials of degree two are involved. This is used to partially resolve the conjecture of T.~Kulczycki, which claims that the second smallest eigenvalue corresponds to an antisymmetric function: we prove that this is the case when either $d \le 2$ and $\alpha \in (0, 2]$, or $d \le 9$ and $\alpha = 1$, and we provide strong numerical evidence %%in the general case
\mk{for $d \le 9$ and general $\alpha \in (0, 2]$}.
\end{abstract}

{\vskip 0.15cm}
 \noindent {\it Keywords}: Fractional Laplace operator, eigenvalues, unit ball, Rayleigh--Ritz method, Aronszajn method, numerical bounds
{\vskip 0.25cm}
 \noindent {\it 2010 Mathematics Subject Classification }: 35P15, 35S05, 65F15

%
%                            ---------- o ----------
%

\section{Introduction and main results}
\label{sec:intro}

For $d \ge 1$ and $\alpha \in (0, 2)$, the \emph{fractional Laplace operator}, or \emph{Riesz fractional derivative}, is defined as
\formula{
%% (-\Delta)^{\alpha/2} f(x) & = \frac{2^\alpha \Gamma(\tfrac{d + \alpha}{2})}{\pi^{d/2} \Gamma(-\tfrac{\alpha}{2})} \lim_{\eps \to 0^+} \int_{\mathbb{R}^d \setminus B(0, \eps)} \frac{f(x) - f(y)}{|x - y|^{d + \alpha}} \, dy
\mk{ (-\Delta)^{\alpha/2} f(x) }&\mk{ = -\frac{2^\alpha \Gamma(\tfrac{d + \alpha}{2})}{\pi^{d/2} |\Gamma(-\tfrac{\alpha}{2})|} \lim_{\eps \to 0^+} \int_{\mathbb{R}^d \setminus B(0, \eps)} \frac{f(y) - f(x)}{|y - x|^{d + \alpha}} \, dy}
}
(see, for example,~\cite{bib:k:lap,bib:l72}). The eigenvalue problem for $(-\Delta)^{\alpha/2}$ in a bounded domain $D \sub \R^d$, with a zero condition in the complement of $D$:
\formula[eq:problem]{
 \begin{cases} (-\Delta)^{\alpha/2} \ph_n(x) = \lambda_n \ph_n(x) & \text{for $x \in D$,} \\ \ph_n(x) = 0 & \text{for $x \notin D$} \end{cases}
}
(here $n = 0, 1, \dots$), has been studied by numerous authors. \MK{For general results, such as existence and basic properties of solutions, we refer the reader to~\cite{bib:bk04,bib:bbkrsv09}. Here we only mention that $\lambda_n$ can be arranged in a non-decreasing unbounded sequence, the fundamental eigenvalue $\lambda_0$ is positive and simple, and $\ph_0$ has a constant sign in $D$. The following general estimate of $\lambda_n$ was proved in~\cite{bib:cs05} (see also~\cite{bib:dm07}): if $D$ is convex, $0 < \alpha \le \beta \le 2$ and $\lambda_n\mk{(\alpha)}$ denotes the sequence of eigenvalues of the problem~\eqref{eq:problem} (arranged in a non-decreasing order) with a given parameter $\alpha$, then
\formula{
 \tfrac{1}{2} (\lambda_n\mk{(\beta)})^{\alpha/\beta} \le \lambda_n\mk{(\alpha)} \le (\lambda_n\mk{(\beta)})^{\alpha/\beta} .
}
This is particularly useful when $\beta = 2$, because $\lambda_n\mk{(2)}$ is known explicitly for many domains. For example, if $D$ is the unit ball, $\lambda_n\mk{(2)}$ is the square of an appropriate zero of the Bessel function. Sharper bounds for $\lambda_n$ are known only when $D$ is a ball and either $n = 0$ (see~\cite{bib:bk04,bib:d12}) or $d = 1$ (see~\cite{bib:bk04,bib:k12}).}

From now on, $D$ denotes the unit ball in $\R^d$ and $\alpha \in (0, 2]$. In a companion paper~\cite{bib:dkk15} we find explicit expressions for $(-\Delta)^{\alpha/2}$ applied to a variety of function. In particular, we find the eigenvalues and eigenfunctions (which turn out to be polynomials) of the operator $f \mapsto (-\Delta)^{\alpha/2}(\omega f)$, where $\omega(x) = (1 - |x|^2)_+^{\alpha/2}$; here and below $a_+ = \max(a, 0)$. This result is stated in Theorem~\ref{th:jacobi} below. In the present article, we use these eigenfunctions to find estimates of $\lambda_n$. The upper bounds follow by the standard Rayleigh--Ritz variational method, while for the lower bounds we use a less-known Aronszajn method of intermediate problems. These are essentially \emph{numerical} methods designed for finding estimates of the eigenvalues of an appropriate variational problem. Nevertheless, the same methods can be used to prove \emph{analytical} bounds for the first few eigenvalues, when matrices and polynomials of small degree are involved.

Before we state our main results, we explain why one can restrict attention to radial eigenfunctions, and this requires some notation. We say that $V$ is a \emph{solid harmonic polynomial} in $\R^d$ of degree $l \ge 0$ if $V$ is a homogeneous polynomial of degree $l$ which is harmonic (that is, $\Delta V(x) = 0$ for all $x \in \R^d$). Solid harmonic polynomials of a given degree $l$ form a finite-dimensional vector space of dimension $M_{d,l} = \tfrac{d+2l-2}{d+l-2} \, \binom{d+l-2}{l}$, and the $L^2$ space over the surface measure on the unit sphere is a direct sum of these spaces over $l \ge 0$ \mk{(see~\cite{bib:abw01,bib:dx14})}. We fix an orthonormal basis of this $L^2$ space, which will be denoted by $\{V_{l,m}\}$, with $l \ge 0$ and $1 \le m \le M_{d,l}$, so that $V_{l,m}$ is a solid harmonic polynomial of degree $l$.

\mk{The solutions of the problem~\eqref{eq:problem} for the unit ball $D$ fall into different symmetry classes, described by solid harmonic polynomials. This fact follows easily from} %%The following result is a direct consequence of 
\emph{Bochner's relation}, which asserts that every Fourier multiplier with radial symbol $m(|\xi|)$ maps a function on $\R^d$ of the form $V(x) f(|x|)$ to a function %%of 
$V(x) g(|x|)$ of the same type, and furthermore a multiplier with symbol $m(|\tilde{\xi}|)$ maps $f(|\tilde{x}|)$ to $g(|\tilde{x}|)$ in dimension $d + 2 l$ (that is, here $\tilde{x}, \tilde{\xi} \in \R^{d + 2 l}$). For more details, see Proposition~3 in~\cite{bib:dkk15}. \mk{As a consequence, each radial eigenfunction, with eigenvalue $\lambda$, of $(-\Delta)^{\alpha/2}$ in a $(d + 2 l)$-dimensional ball gives rise to $M_{d,l}$ non-radial (unless $l = 0$) linearly independent eigenfunctions, with the same eigenvalue $\lambda$, of $(-\Delta)^{\alpha/2}$ in a $d$-dimensional ball. This is formally stated in the following result.}

\begin{proposition}
\label{prop:bochner}
Let $\ph_{d,n}(|x|)$ and $\lambda_{d,n}$ denote the sequence of all eigenfunctions, and the corresponding eigenvalues, which are \emph{radial} solutions of the problem~\eqref{eq:problem} for the unit ball~\mk{$D \subseteq \R^d$}. We assume that $\lambda_{d,n}$ are arranged in a non-decreasing order (with respect to $n$). Then the functions $V_{l,m}(x) \ph_{d + 2 l, n}(|x|)$, where $l \ge 0$, $1 \le m \le M_{d,l}$ and $n \ge 0$, form a complete orthogonal system of solutions of the problem~\eqref{eq:problem}, with corresponding eigenvalues $\lambda_{d + 2 l, n}$.
\end{proposition}

In particular, the sequence $\lambda_n$ can be obtained by rearranging in a non-decreasing way the numbers $\lambda_{d + 2 l, n}$, with $l \ge 0$ and $n \ge 0$, each repeated $M_{d,l}$ times. For this reason in the remaining part of the article we restrict our attention to radial functions, and so we will no longer need harmonic polynomials and the parameter $l$.

The following two theorems are the main results of this article. The first one provides a numerical scheme for the estimates of $\lambda_n$. The other one is an interesting corollary, which partially resolves the conjecture of T.~Kulczycki. In order to state these results, first we need to introduce some \mk{notation}. We denote by $A^{(N)}$ and $B^{(N)}$ the $N \times N$ matrices having entries 
\formula{
 A_{m,n} & = \delta_{m,n} \, \frac{\mk{2^{\alpha}} \pi^{d/2} \Gamma(\tfrac{d}{2}) (\Gamma(\tfrac{\alpha}{2} + n + 1))^2}{(\tfrac{d + \alpha}{2} + 2 n) (\Gamma(\tfrac{d}{2} + n))^2} \, , \\
 B_{m,n} & = \frac{\mk{\pi^{d/2}} \Gamma(\alpha + 1) \Gamma(\tfrac{d}{2}) \Gamma(\tfrac{d}{2} + m + n) \Gamma(\tfrac{\alpha}{2} + m + 1) \Gamma(\tfrac{\alpha}{2} + n + 1)}{\Gamma(\tfrac{\alpha}{2} + m - n + 1) \Gamma(\tfrac{\alpha}{2} + n - m + 1) \Gamma(\tfrac{d}{2} + m) \Gamma(\tfrac{d}{2} + n) \Gamma(\tfrac{d}{2} + m + n + 1 + \alpha)} \, ,
}
with $0 \le m,n < N$ (here and below, $\delta_{n,n} = 1$ and $\delta_{m,n} = 0$ when $m \ne n$). We also define
\formula{
 \mu_n & = \frac{2^\alpha \Gamma(\tfrac{\alpha}{2} + n + 1) \Gamma(\tfrac{d + \alpha}{2} + n)}{n! \, \Gamma(\tfrac{d}{2} + n)} \, , \\
 \sigma_n & = \frac{\pi^{d/2} n! \Gamma(\tfrac{d}{2}) \Gamma(\tfrac{\alpha}{2} + n + 1)}{(\tfrac{d + \alpha}{2} + 2 n) \Gamma(\tfrac{d}{2} + n) \Gamma(\tfrac{d + \alpha}{2} + n)} \, , \\
 I_{m,n} & = \int_{%%B(0, 1)
\mk{D}
} \frac{(P_m(x) - P_{m+1}(x)) (P_n(x) - P_{n+1}(x))}{(1 - |x|^2)^{-\alpha/2} - 1} \, dx , \\
 P_n(x) & = \frac{(-1)^n n! \Gamma(\tfrac{d}{2})}{\Gamma(\tfrac{d}{2} + n)} \, P^{(\alpha/2,d/2-1)}_n(2 |x|^2 - 1)  = {_2F_1}(-n, \tfrac{d + \alpha}{2} + n; \tfrac{d}{2}; |x|^2) .
}
Here \mk{$D$ is the unit ball,} $P^{(\alpha,\beta)}_n$ is the Jacobi polynomial, and ${_2F_1}$ is the \mk{Gauss's} hypergeometric function. For the last equality, see formula~8.962.1 in~\cite{bib:gr07}.

\begin{theorem}
\label{th:eigenvalues}
Let $d \ge 1$, $N \ge 0$ and $0 < \alpha \le 2$. Denote by $\lambda_{d,n}$, with $n \ge 0$, the non-decreasing sequence of the eigenvalues corresponding to radial solutions of the problem~\eqref{eq:problem} for the unit ball \mk{$D \subseteq \R^d$}. Then
\formula[eq:eigenbounds]{
 \mk{\lambdal_{d,n}^{(N)}} \le \lambda_{d,n} \le \mk{\lambdau_{d,n}^{(N)}} ,
}
where $\mk{\lambdal_{d,n}^{(N)}}$ and $\mk{\lambdau_{d,n}^{(N)}}$ are defined as follows:
\begin{itemize}
\item[(i)]
The numbers $\mk{\lambdau_{d,n}^{(N)}}$, with $0 \le n < N$, are the solutions $\lambda$, arranged in a nondecreasing order, of the $N \times N$ matrix eigenvalue problem $A^{(N)} x = \lambda B^{(N)} x$. For $n \ge N$, we let $\mk{\lambdau_{d,n}^{(N)}} = \infty$.
\item[(ii)]
The sequence $\mk{\lambdal_{d,n}^{(N)}}$, with $n \ge 0$, %%consists of the $N + 1$ zeroes of the polynomial
\mk{is the nondecreasing rearrangement of the sequence, whose first $N + 1$ terms are the $N + 1$ zeroes of the polynomial}
\formula{
 w^{(N)}(\lambda) & = \expr{\prod_{n = 0}^N (\mu_n - \lambda)} \det W^{(N)}(\lambda) \mk{,}
}
%%as well as the numbers $\mu_n$ for $n > N$, rearranged in a nondecreasing infinite sequence. 
\mk{and the remaining terms are the numbers $\mu_n$, with $n \ge N + 1$.} 
Here the entries of \mk{the matrix} $W^{(N)}(\lambda)$ are given by
\formula{
 W_{m,n}(\lambda) & = I_{m,n} + \frac{\lambda \delta_{m,n} \sigma_n}{\mu_m - \lambda} - \frac{\lambda \delta_{m+1,n} \sigma_n}{\mu_{m+1} - \lambda} - \frac{\lambda \delta_{m,n+1} \sigma_{n+1}}{\mu_m - \lambda} + \frac{\lambda \delta_{m+1,n+1} \sigma_{n+1}}{\mu_{m+1} - \lambda} \, ,
}
with $0 \le m,n < N$.
\end{itemize}
\end{theorem}

\mk{We emphasize that quite often the zeroes of the polynomial $w^{(N)}(\lambda)$ are interlaced with the numbers $\mu_n$, $n \ge N + 1$. For example, depending on the parameters $d$ and $\alpha$, the lower bound $%%\lambdal_1(d, 1)
\mk{\lambdal_{d,1}^{(1)}}
$ is equal either to the larger zero of $w^{(1)}(\lambda)$ or to $\mu_2$, see Figure~\ref{fig:gap}. Thus, nondecreasing rearrangement of the sequence of lower bounds in part~(ii) of Theorem~\ref{th:eigenvalues} is essential.}

\mk{Observe that for $N = 0$ the estimate~\eqref{eq:eigenbounds} reduces to $\lambda_{d,n} \ge \mu_n$. For $N > 0$, t}he expression for \mk{$\lambdal_{d,n}^{(N)}$} is rather complicated, but as we will see below both \mk{lower and upper bounds of Theorem~\ref{th:eigenvalues}} are well-suited for %%both 
numerical calculations and symbolic manipulation.

\begin{theorem}\label{th:antisymmetric}
Let $1 \le d \le 9$ and $\alpha = 1$, or \mk{$d \in \{1, 2\}$} and $0 < \alpha \le 2$. %%Denote by $\lambda_n$, with $n \ge 0$, the non-decreasing sequence of the eigenvalues 
\mk{Let $\lambda_1$ be the second smallest eigenvalue}
of the problem~\eqref{eq:problem} for the unit ball $D$. Then the eigenfunctions corresponding to $\lambda_1$ are antisymmetric, that is, they satisfy the relation $\ph(-x) = -\ph(x)$.
\end{theorem}

\mk{By Proposition~\ref{prop:bochner}, the solutions of~\eqref{eq:problem} for the unit ball $D$ are the numbers $\lambda_{d+2l,n}$, where $l, n \ge 0$. By definition, $\lambda_{d+2l,n}$ is nondecreasing in $n \ge 0$, and $\lambda_{d+2l,n} > \lambda_{d+2l,0}$ when $n > 0$. Furthermore, $\lambda_{d+2l,0}$ is strictly increasing in $l \ge 0$ (see Section~\ref{sec:eigenvalues}). Thus $\lambda_{d,0}$ is the smallest eigenvalue of the problem~\eqref{eq:problem}, and the only possible values of $\lambda_1$ are $\lambda_{d+2,0}$ and $\lambda_{d,1}$.}

\mk{In order to prove Theorem~\ref{th:antisymmetric}, it suffices to show that $\lambda_{d+2,0} < \lambda_{d,1}$. Indeed, then $\lambda_{d+2l,n} = \lambda_1$ only if $l = 1$ and $n = 0$ (by the argument used in the previous paragraph), and thus the eigenfunctions of~\eqref{eq:problem} with eigenvalue $\lambda_1$ are of the form $\ph(x) = V(x) f(|x|)$ for a solid harmonic polynomial $V$ of degree $l = 1$. This means that $V$ is a linear function, and so $\ph(-x) = V(-x) f(|x|) = -V(x) f(|x|) = -\ph(x)$, as desired.}

As mentioned above, \mk{the inequality $\lambda_{d+2,0} < \lambda_{d,1}$ (equivalent to Theorem~\ref{th:antisymmetric})} follows by evaluating analytically the bounds of Theorem~\ref{th:eigenvalues} for small values of $N$. \mk{More precisely, we prove in Section~\ref{sec:analytic} that $%%\lambdau_0(d + 2, 2)
\mk{\lambdau_{d+2,0}^{(2)}}
 < %%\lambdal_1(d, 1)
\mk{\lambdal_{d,1}^{(1)}}
$.}

Apparently the above bounds for $\lambda_{d+2,0}$ and $\lambda_{d,1}$ are sharp enough to assert that $\lambda_{d+2,0} < \lambda_{d,1}$ for all $d \le 9$ and $\alpha \in (0, 2]$, see Figure~\ref{fig:gap}; nevertheless, we were only able to overcome technical difficulties when $\alpha = 1$ or $d \le 2$. We remark that numerical simulations clearly indicate that $\lambda_{d+2,0} < \lambda_{d,1}$ for general $\alpha \in (0, 2]$ and $d \ge 1$ (which is a well-known result for $\alpha = 2$), in agreement with T.~Kulczycki's conjecture.

% 2.754754802175552132849850013664198925477309320140631762168911515550798597576252002718615935864628309 N=10
% 2.754754742234920614282808076144723458690479572355488147712717974548940817224651765477768570652942705 N=40
% 2.754754742220396752854418549207689857167461359351774738036489545197964292189010251904038176580899099 N=50
% 2.754754741530880138141428572005563453922 N=40
% 2.754754735556359990819655382456281212378 N=25
% 2.754754722946872606054314123552068741047 N=20
% 2.754754668215464695824105907414500634084 N=15
% 2.754754279982159694240590945634984070268 N=10

% 2.006119032946854741371037992186015766686653248665494455616889107689760861242506264221563449664601211
% 2.006119032561155673254258236782829460068

\MK{Theorem~\ref{th:antisymmetric} was known only for $d = 1$ and $\alpha \in [1, 2]$: the case $\alpha = 1$ was solved in~\cite{bib:bk04}, while an extension to $\alpha \in [1, 2]$ is one of the results of~\cite{bib:k12}. In both articles the proof reduces to sufficiently sharp bounds for $\lambda_{d,1}$ and $\lambda_{d+2,0}$.}

%\MK{Denote
%\formula{
% \Lambda(d) & = \frac{(2 d^2 + 4 \alpha d + 8 d - \alpha) \Gamma(\tfrac{\alpha}{2} + 1) \Gamma(\tfrac{\alpha}{2} + 2) \Gamma(\tfrac{d + \alpha}{2}) \Gamma(\tfrac{d}{2} + \alpha + 2)}{2^{3 - \alpha} \Gamma(\tfrac{d}{2} + 1) \Gamma(\alpha+2) \Gamma(\tfrac{d+\alpha}{2}+3)} .
%}
%In the proof of Theorem~\ref{th:antisymmetric} we show the following estimates:
%\begin{enumerate}[label=(\alph*)]
%\item if $\alpha = 1$, then $\lambda_{d,0} < \tfrac{3 \pi}{16} (d + 3)$;
%\item if $\alpha = 1$ and $d \le 9$, then $\lambda_{d,1} > \tfrac{3 \pi}{16} (d + 5)$;
%\item $\lambda_{d,0} < \Lambda(d)$;
%\item if $d = 1$ or $d = 2$, then $\lambda_{d,1} > \Lambda(d + 2)$,
%\end{enumerate}}

The numerical scheme of Theorem~\ref{th:eigenvalues} extends the one studied in~\cite{bib:kkms10}, where $d = 1$ and $\alpha = 1$ was considered. In this case the corresponding explicit expressions follow easily by harmonic extension and conformal mapping, and the Aronszajn method (called Weinstein--Aronszajn in this case) simplifies significantly.

\mk{According to numerical calculations, a}s long as $d$ is not very large, the %%
\mk{rate of} 
convergence of both upper and lower bounds %%
\mk{to the correct values of $\lambda_{d,n}$} 
is rather fast, at least when compared to other known methods (\cite{bib:dz14,bib:gm15,bib:zg14,bib:zrk07} for $d = 1$ and~\cite{bib:k12} for general $d$), see Tables~\ref{tab:eigenvalues0}--\ref{tab:eigenvalues2} and Figures~\ref{fig:eigenvalues}--\ref{fig:eigenvaluesdim}. For example, using just $40 \times 40$ matrices, one finds that $\lambda_{d,0} = 2.0061190327 \pm 3 \cdot 10^{-10}$ for $d = 2$ and $\alpha = 1$, and $\lambda_{d,0} = 2.754754742 \pm 10^{-9}$ for $d = 3$ and $\alpha = 1$.

The upper bounds are given as eigenvalues of a well-conditioned matrix and thus they can be easily computed in a numerically stable way. Lower bounds are more problematic, they require numerical evaluation of roots of a polynomial given by the determinant of a matrix with a parameter. Due to accumulation of numerical errors and singularities of the entries $W_{m,n}(\lambda)$, all calculations should be carried out with additional precision; see~\cite{bib:ws72} for a detailed discussion of the Aronszajn method in the classical context. %[[need to consult~\cite{bib:ws72} myself]]

As remarked above, our results are based on explicit expressions for the eigenvalues and eigenfunctions of the operator $f \mapsto (-\Delta)^{\alpha/2}(\omega f)$, where $\omega(x) = (1 - |x|^2)_+^{\alpha/2}$, found in~\cite{bib:dkk15}. Roughly speaking, the result states that for any polynomial $P$, the function $(-\Delta)^{\alpha/2}(\omega P)$ is equal in the unit ball to another polynomial of the same degree. \mk{This phenomenon was first observed in~\cite{bib:bik11,bib:d12} for radial (or radial times linear) functions, and extended to arbitrary polynomials in~\cite{bib:dkk15}.} Below we recall the result, restricted to the case of radial functions, and with a modified constant in the definition of $P_n$, which is more suitable for calculations.

\begin{theorem}[{Theorem~3 in~\cite{bib:dkk15}}]\label{th:jacobi}
Let $d \ge 1$, $n \ge 0$ and $\alpha \in (0, 2]$. Define $P_n$ and $\mu_n$ as in Theorem~\ref{th:eigenvalues}, and let $\omega(x) = (1 - |x|^2)_+^{\alpha/2}$. Then for $x$ \mk{in the unit ball in $\R^d$},
\formula[eq:jacobi]{
 (-\Delta)^{\alpha/2} (\omega P_n)(x) & = \mu_n P_n(x) .
}
\end{theorem}

We remark that the polynomials $P_n$ form a complete orthogonal system in $L^2_\rad(D, \omega)$, the weighted $L^2$ space of radial functions in $D$, with weight function $\omega$. A similar system for the full $L^2(D, \omega)$ space is given in the original statement in~\cite{bib:dkk15}.

\medskip

We conclude the introduction with the outline of the article. \MK{In Section~\ref{sec:notation} we introduce additional notation related to the polynomials $P_n$, and prove some preliminary identities and estimates. Theorem~\ref{th:eigenvalues} is proved in Section~\ref{sec:eigenvalues}, while Section~\ref{sec:analytic} contains the proof of Theorem~\ref{th:antisymmetric}.}
%
%                            ---------- o ----------
%

\section{Notation}
\label{sec:notation}

We use the notation of Theorems~\ref{th:eigenvalues} and~\ref{th:jacobi}. Recall that $\omega(x) = (1 - |x|^2)_+^{\alpha/2}$, and that $\sqrt{\omega} \, P_n$ form a complete orthogonal set in the space of radial $L^2(D)$ functions, denoted here and below by $L^2_\rad(D)$. We let
\formula[eq:pnnorm]{
 \sigma_n = \|\sqrt{\omega} \, P_n\|_2^2 & = \frac{2 \pi^{d/2}}{\Gamma(\tfrac{d}{2})} \biggl(\frac{n! \Gamma(\tfrac{d}{2})}{\Gamma(\tfrac{d}{2} + n)}\biggr)^{\!\!2} \int_0^1 r^{d-1} (1 - r^2)^{\alpha/2} (P^{(\alpha/2, d/2 - 1)}_n(2 r^2 - 1))^2 dr \\
 & = \frac{\pi^{d/2} (n!)^2 \Gamma(\tfrac{d}{2})}{\mk{2^{(d + \alpha)/2}} (\Gamma(\tfrac{d}{2} + n))^2} \int_{-1}^1 (1 + s)^{d/2-1} (1 - s)^{\alpha/2} (P^{(\alpha/2, d/2 - 1)}_n(s))^2 ds \\
 & = \frac{\pi^{d/2} n! \Gamma(\tfrac{d}{2}) \Gamma(\tfrac{\alpha}{2} + n + 1)}{(\tfrac{d + \alpha}{2} + 2 n) \Gamma(\tfrac{d}{2} + n) \Gamma(\tfrac{d + \alpha}{2} + n)} \, ,
}
see formula~7.391.1 in~\cite{bib:gr07}. We also define
\formula[eq:pnm]{
 \pi_{m,n} & = \langle \omega P_m, \omega P_n \rangle \\
 & = \frac{(-1)^{m + n} \pi^{d/2} m! n! \Gamma(\tfrac{d}{2})}{\mk{2^{d/2 + \alpha}} \Gamma(\tfrac{d}{2} + m) \Gamma(\tfrac{d}{2} + n)} \int_{-1}^1 (1 + s)^{d/2-1} (1 - s)^\alpha P^{(\alpha/2, d/2 - 1)}_n(s) P^{(\alpha/2, d/2 - 1)}_m(s) ds \\
 & = \frac{\mk{\pi^{d/2}} \Gamma(\alpha + 1) \Gamma(\tfrac{d}{2}) \Gamma(\tfrac{d}{2} + m + n) \Gamma(\tfrac{\alpha}{2} + m + 1) \Gamma(\tfrac{\alpha}{2} + n + 1)}{\Gamma(\tfrac{\alpha}{2} + m - n + 1) \Gamma(\tfrac{\alpha}{2} + n - m + 1) \Gamma(\tfrac{d}{2} + m) \Gamma(\tfrac{d}{2} + n) \Gamma(\tfrac{d}{2} + \alpha + m + n + 1)} \, ,
}
see formula~16.4(1\mk{7}) in~\cite{bib:e54} (for $n = m$ this is formula~7.391.6 in~\cite{bib:gr07}).
%Finally, we have
%\formula[eq:pnrec]{
% |x|^2 P_n(x) & = -\alpha_n P_{n+1}(x) + \beta_n P_n(x) - \gamma_n P_{n-1}(x) ,
%}
%where
%\formula[eq:pnrecc]{
% \alpha_n & = \frac{(\tfrac{d}{2} + n) (\tfrac{d + \alpha}{2} + n)}{(\tfrac{d + \alpha}{2} + 2 n) (\tfrac{d + \alpha}{2} + 2 n + 1)} \, , \\
% \beta_n & = \frac{\tfrac{d + \alpha}{2} (\tfrac{d}{2} + 2 n) - \tfrac{d}{2} + 2 n^2}{(\tfrac{d + \alpha}{2} + 2 n - 1) (\tfrac{d + \alpha}{2} + 2 n + 1)} \, , \\
% \gamma_n & = \frac{n (\tfrac{\alpha}{2} + n)}{(\tfrac{d + \alpha}{2} + 2 n - 1) (\tfrac{d + \alpha}{2} + 2 n)} \, .
%}
Finally, in the proof of Theorem~\ref{th:antisymmetric}, the integral
\formula[eq:int:i]{
 I = I_{0,0} & = \int_D \frac{(1 - P_1(x))^2}{(1 - |x|^2)^{-\alpha/2} - 1} \, dx
}
plays an important role. Since $P_1(x) = 1 - \tfrac{d + \alpha + 2}{d} |x|^2$, we have
\formula[eq:int:i2]{
 I & = \frac{(d + \alpha + 2)^2}{d^2} \int_D \frac{|x|^4}{(1 - |x|^2)^{-\alpha/2} - 1} \, dx \\
 & = \frac{(d + \alpha + 2)^2 \pi^{d/2}}{d^2 \Gamma(\tfrac{d}{2})} \int_0^1 \frac{s^2}{(1 - s)^{-\alpha/2} - 1} \, s^{d/2 - 1} ds \\
 & = \frac{(d + \alpha + 2)^2 \pi^{d/2}}{d^2 \Gamma(\tfrac{d}{2})} \int_0^1 \frac{t^{\alpha/2} (1 - t)^{d/2 + 1}}{1 - t^{\alpha/2}} \, dt \\
 & = \frac{(d + \alpha + 2)^2 \pi^{d/2}}{d^2 \Gamma(\tfrac{d}{2})} \sum_{j = 1}^\infty \int_0^1 t^{j \alpha/2} (1 - t)^{d/2 + 1} dt \\
 & = \frac{(d + \alpha + 2)^2 \pi^{d/2}}{d^2 \Gamma(\tfrac{d}{2})} \sum_{j = 1}^\infty \frac{\Gamma(\tfrac{j \alpha}{2} + 1) \Gamma(\tfrac{d}{2} + 2)}{\Gamma(\tfrac{j \alpha}{2} + \tfrac{d}{2} + 3)} \\
 & = \frac{(d + 2) (d + \alpha + 2)^2 \pi^{d/2}}{4 d} \sum_{j = 1}^\infty \frac{\Gamma(\tfrac{j \alpha}{2} + 1)}{\Gamma(\tfrac{j \alpha}{2} + \tfrac{d}{2} + 3)} \, .
}
By a direct calculation,
\formula[eq:int:i:est]{
 \sigma_0 + \sigma_1 & =  \frac{ (d + 2) (d + \alpha + 2)^2 \pi^{d/2} \Gamma(\tfrac{\alpha}{2} + 1)}{4 d \, \Gamma(\tfrac{d + \alpha}{2} + 3)} < I .
}
The lower bounds for $I$ can be obtained by truncating the series in~\eqref{eq:int:i2}. We present two different upper bounds. The first one uses %%concavity 
\mk{convexity} 
of $\log \Gamma(z)$: we have
\formula{
 \log \Gamma(s + \tfrac{d}{2} + 3) - \log \Gamma(s + 1) \ge (\tfrac{d}{2} + 2) (\log \Gamma(s + 2) - \log \Gamma(s + 1)) ,
}
and therefore
\formula[eq:zeta]{
 I & \le \frac{(d + 2) (d + \alpha + 2)^2 \pi^{d/2}}{4 d} \expr{\sum_{j = 1}^J \frac{\Gamma(\tfrac{j \alpha}{2} + 1)}{\Gamma(\tfrac{j \alpha}{2} + \tfrac{d}{2} + 3)} + \sum_{j = J + 1}^\infty (1 + \tfrac{j \alpha}{2})^{-d/2 - 2}} \\
 & \le \frac{(d + 2) (d + \alpha + 2)^2 \pi^{d/2}}{4 d} \expr{\sum_{j = 1}^J \frac{\Gamma(\tfrac{j \alpha}{2} + 1)}{\Gamma(\tfrac{j \alpha}{2} + \tfrac{d}{2} + 3)} + \int_J^\infty (1 + \tfrac{s \alpha}{2})^{-d/2 - 2} ds} \\
 & = \frac{(d + 2) (d + \alpha + 2)^2 \pi^{d/2}}{4 d} \expr{\sum_{j = 1}^J \frac{\Gamma(\tfrac{j \alpha}{2} + 1)}{\Gamma(\tfrac{j \alpha}{2} + \tfrac{d}{2} + 3)} + \frac{4}{(d + 2) \alpha (1 + \tfrac{J \alpha}{2})^{d/2 + 1}}} .
}
We remark that a similar method can be used to find numerical estimates of $I_{m,n}$ for general $m$ and $n$. A different upper bound, which will be used in Section~\ref{subsec:d2}, is obtained by estimating the function under the integral in the definition of $I$: the function $t^{\alpha/2}$ is concave, and hence $(1 - t^{\alpha/2}) / (1 - t) \ge \tfrac{\alpha}{2}$. Therefore,
\formula[eq:J]{
 I & = \frac{(d + \alpha + 2)^2 \pi^{d/2}}{d^2 \Gamma(\tfrac{d}{2})} \int_0^1 \frac{t^{\alpha/2} (1 - t)^{d/2 + 1}}{1 - t^{\alpha/2}} \, dt \\
 &\leq \frac{(d + \alpha + 2)^2 \pi^{d/2}}{ d^2 \Gamma(\tfrac{d}{2})}
\frac{2}{\alpha} B\Big(\frac{d}{2}+1, \frac{\alpha}{2}+1\Big) 
= \frac{2(d+\alpha+2)\sigma_0}{\alpha d} . %=: J.
}
%\MK{In Section~\ref{subsec:d2}, we denote the right-hand side by $J$.}
%The advantage of this estimate is that it is accurate when $\alpha\to 0^+$ or $\alpha=2$.

\section{Numerical scheme}
\label{sec:eigenvalues}

In this section we prove Theorem~\ref{th:eigenvalues}. We begin with the well-known variational characterisation of the eigenvalues $\lambda_{d,n}$, and then describe the application of Rayleigh--Ritz and Aronszajn methods. Noteworthy, extensions to $\alpha > 2$ are possible, and our estimates are in fact valid for all $\alpha > 0$. However, we will restrict our attention to the more important and much better understood case $\alpha \in (0, 2]$. 

Let $\form(f, g)$, $f, g \in \dom(\form)$, denote the Dirichlet form associated with $(-\Delta)^{\alpha/2}$ in $D$ (restricted to the space $L^2_\rad(D)$ of radial functions), with Dirichlet condition in $\R^d \setminus D$. That is, for $\alpha \in (0, 2)$, we have
\formula{
 \form(f, g) & = \frac{2^{\alpha - 1} \Gamma(\tfrac{d + \alpha}{2})}{\pi^{d/2} %%\Gamma(\tfrac{\alpha}{2})
\mk{|\Gamma(-\tfrac{\alpha}{2})|}
} \int_{\R^d} \int_{\R^d} \frac{(f(x) - f(y)) (\overline{g(x)} - \overline{g(y)})}{|x - y|^{d + \alpha}} \, dx \, dy , \\
 \dom(\form) & = \set{ f \in L^2_\rad(D) : \int_{\R^d} \int_{\R^d} \frac{|f(x) - f(y)|^2}{|x - y|^{d + \alpha}} \, dx \, dy < \infty } ,
}
where all functions $f \in L^2_\rad(D)$ are extended to $\R^d$ so that $f(x) = 0$ for $x \notin D$, while for $\alpha = 2$ we have the usual energy form
\formula{
 \form(f, g) & = \int_D \nabla f(x) \cdot \overline{\nabla g(x)} dx ,
}
defined on the Sobolev space $W^{1,2}_0(D)$. For further information about the above Dirichlet forms and related objects, we refer the reader to~\cite{bib:dpv12,bib:s70}. A general account on Dirichlet forms can be found in~\cite{bib:fot11}.

Define the Rayleigh quotient
\formula{
 Q(f) & = \frac{\form(f, f)}{\|f\|_2^2}
}
for $f \in \dom(\form)$, $f \ne 0$, and $Q(f) = 0$ for $f = 0$. By the variational principle, the non-decreasing sequence $\lambda_{%%0
\mk{d}
,n}$, $n \ge 0$, of the eigenvalues of $(-\Delta)^{\alpha/2}$ in $D$, restricted to the subspace $L^2_\rad(D)$ of radial functions, is equal to
\formula{
 \lambda_{%%0
\mk{d}
,n} & = \inf \set{\sup \set{Q(f) : f \in E } : \text{$E$ is a $(n + 1)$-dimensional subspace of $\dom(\form)$}} .
}
We note that $\lambda_{%%0
\mk{d}
,0}$ \mk{strictly} increases with the dimension $d$. Indeed, %%using the above variational formula, one can easily prove that, given $\eps > 0$, the smallest eigenvalue for a $(d+1)$-dimensional domain $D \times [-R, R]$ (where $D$ is the unit ball in $\R^d$) for $R$ large enough is greater than $\lambda_{%%0
%%\mk{d}
%%,0} - \eps$, where $\lambda_{%%0
%%\mk{d}
%%,0}$ is the smallest eigenvalue for $D$. By domain monotonicity, the smallest eigenvalue for the $(d+1)$-dimensional unit ball is therefore greater than $\lambda_{%%0
%%\mk{d}
%%,0} - \eps$. Since $\eps > 0$ is arbitrary, our claim is proved. 
\mk{let $f \in W^{1,2}_0(D \times [-1, 1])$, where $D$ is the unit ball in $\R^d$ (so $f$ is defined on a domain in $\R^{d+1}$), and let $g(x) = \int_{-1}^1 f(x,y) dy$. Then $g \in W^{1,2}_0(D)$, and it is easy to see that $Q(f) \ge Q(g)$ (where the two Rayleigh quotients are defined on $D \times [-1, 1]$ and $D$, respectively), and thus $Q(f) \ge \lambda_{d,0}$. It follows that the smallest eigenvalue of the problem~\eqref{eq:problem} in $D \times [-1, 1]$ is not less than $\lambda_{d,0}$. By domain monotonicity, $\lambda_{d+1,0} > \lambda_{d,0}$, that is, $\lambda_{d,0}$ is indeed strictly increasing in $d \ge 1$.} We remark that monotonicity in $d$ of $\lambda_{d,n}$ for $n \ge 1$ appears to be an open question.

\subsection{Upper bounds}

For the upper bounds for $\lambda_{%%0
\mk{d}
,n}$, we use the standard Rayleigh--Ritz method. Let $f_n = \omega P_n$ for $n \ge 0$. Then $f_n \in \dom(\form)$ and $\form(f_m, f_n) = \langle f_m, (-\Delta)^{\alpha/2} f_n \rangle$, with $(-\Delta)^{\alpha/2} f_n$ defined pointwise. \MK{The proof of this fact is standard, but somewhat complicated: if $G_D$ denotes the Green operator for $(-\Delta)^{\alpha/2}$ in the unit ball (for more information about the Green operator in this context, see, for example, \cite{bib:b99,bib:bb99,bib:l72,bib:r38a}), then $g = f_n - G_D (-\Delta)^{\alpha/2} f_n$ is continuous in $\R^d$, zero outside $D$ and it satisfies $(-\Delta)^{\alpha/2} g(x) = 0$ for all $x \in D$. By the results of~\cite{bib:bb99}, $g$ is everywhere zero, and hence $f_n = G_D (-\Delta)^{\alpha/2} f_n$\mk{. In particular, $f_n$ belongs to the $L^2$ domain of $(-\Delta)^{\alpha/2}$ in $D$, and thus $f_n \in \dom(\form)$ and $\form(f_m, f_n) = \langle f_m, (-\Delta)^{\alpha/2} f_n \rangle$, as desired.}}

It follows that
\formula{
 \form(f_m, f_n) & = \langle f_m, (-\Delta)^{\alpha/2} f_n \rangle = \mu_n \langle \omega P_m, P_n \rangle = \delta_{m,n} \mu_n \sigma_n ,
}
where $\delta_{m,n} = 1$ if $n = m$, $\delta_{m,n} = 0$ otherwise. On the other hand, $\langle f_m, f_n \rangle = \pi_{m,n}$. Let $\mk{\lambdau_{d,n}^{(N)}}$, $0 \le n < N$, be the non-decreasing sequence of the solutions of the $N \times N$ matrix eigenvalue problem $A^{(N)} x = \lambda B^{(N)} x$, where the entries of $A^{(N)}$ and $B^{(N)}$ are given by $A_{m,n} = \delta_{m,n} \mu_n \sigma_n$ and $B_{m,n} = \pi_{m,n}$, with $0 \le m,n < N$. By the variational characterisation, we have $\lambda_{%%0
\mk{d}
,n} \le \mk{\lambdau_{d,n}^{(N)}}$ when $0 \le n < N$. This is precisely the first part of Theorem~\ref{th:eigenvalues}.

\subsection{Lower bounds}

The lower bounds for $\lambda_{%%0
\mk{d}
,n}$ are found using the Aronszajn method of intermediate problems, see e.g.~\cite{bib:b87}. Since this is not as well-known as the Rayleigh--Ritz method, we provide a short general description. Consider the eigenvalue problem $\A f = \lambda \B f$ for non-negative definite operators $\A$, $\B$. In our case, $\A$ is the fractional Laplace operator $(-\Delta)^{\alpha/2}$ in the unit ball $D$, and $\B$ is the identity operator. Suppose furthermore, that the solutions of a different eigenvalue problem $\A f = \lambda \B^{(0)}%%
 \mk{f}
$, the so-called \emph{base eigenvalue problem}, are known explicitly. Here $\B^{(0)}$ is considered to be a perturbation of $\B$, and $\tilde{\B} = \B^{(0)} - \B$ is assumed to be non-negative definite. In our case, $\B^{(0)} f = \omega^{-1} f$ (here and below we understand that $\omega^{-1}(x) = (1 - |x|^2)^{-\alpha/2}$ for $x \in D$; we will never use this symbol for $x \notin D$), and $f_n = \omega P_n$ are the eigenfunctions of the base problem, with corresponding eigenvalues $\mu_n$. By the variational characterisation, we have the basic lower bound $\lambda_{%%0
\mk{d}
,n} \ge \mu_n$ for $n \ge 0$.

Improved lower bounds for $\lambda_{%%0
\mk{d}
,n}$ are found by solving the \emph{intermediate eigenvalue problem} $\A f = \lambda \B^{(N)} f$, where the intermediate operator is defined by
\formula{
 \B^{(N)} f & = \B^{(0)} f - \sum_{m = 0}^{N - 1} \sum_{n = 0}^{N - 1} \langle f, \tilde{\B} g_m \rangle G_{m,n}^{(N)} \tilde{\B} g_n ;
}
here $g_n$ is a sequence of appropriately chosen linearly independent test functions, and $G_{m,n}^{(N)}$ are the entries of the matrix inverse to the Gramian matrix
\formula{
 (\langle g_m, \tilde{\B} g_n \rangle : 0 \le m,n < N) .
}
More precisely, if $\mk{\lambdal_{d,n}^{(N)}}$ denotes the non-decreasing sequence of eigenvalues of the intermediate problem, then $\lambda_{%%0
\mk{d}
,n} \ge \mk{\lambdal_{d,n}^{(N)}}$.

Note that the intermediate problem with $N = 0$ is simply the base problem, with $%%\lambdal_n(d, 0)
\mk{\lambdal_{d,n}^{(0)}}
 = \mu_n$\mk{. Furthermore,} $\B^{(N)}$ is a finite rank perturbation of~$\B^{(0)}$. \mk{More precisely, $f^{(N)} = \sum_{m = 0}^{N - 1} \sum_{n = 0}^{N - 1} \langle f, \tilde{\B} g_m \rangle G_{m,n}^{(N)} g_n$ is the projection of $f$ onto the linear span of $g_n$, $0 \le n < N$ (in fact, an orthogonal projection with respect to the quadratic form of $\tilde{\B}$). Therefore, $\B^{(0)} f - \B^{(N)} f = \tilde{\B} f^{(N)}$ is a projection of $\tilde{\B} f$ onto the linear span of $\tilde{\B} g_n$, $0 \le n < N$. Hence, $\B^{(0)} f - \B^{(N)} f$ converges $\tilde{\B} f$ as $N \to \infty$ (under appropriate assumptions on the choice of $g_n$), and so $\B^{(N)} f$ converges to $\B f$. This can be proved formally and used to show that the eigenvalues of the intermediate problems converge as $N \to \infty$ to the eigenvalues of the original problem, but we will not need this result.}

\mk{The intermediate eigenvalue problem $\A f = \lambda \B^{(N)} f$ can be written as
\formula{
 (\A - \lambda \B^{(0)}) f + \lambda \sum_{m = 0}^{N - 1} \sum_{n = 0}^{N - 1} \langle f, \tilde{\B} g_m \rangle G_{m,n}^{(N)} \tilde{\B} g_n = 0 .
}
Fix $\lambda$ such that $\A - \lambda \B^{(0)}$ is invertible. In this case the above equation reads
\formula{
 f + \lambda \sum_{m = 0}^{N - 1} \sum_{n = 0}^{N - 1} \langle f, \tilde{\B} g_m \rangle G_{m,n}^{(N)} (\A - \lambda \B^{(0)})^{-1} \tilde{\B} g_n = 0.
}
Assuming that $f$ is a linear combination of $g_m$, $0 \le m < N$, and taking the inner product with $\tilde{\B} g_n$, $0 \le n < N$, we obtain a system of linear equations. The coefficients of these equations form the $N \times N$ Weinstein--Aronszajn matrix} 
%%The eigenvalues of the intermediate problem are described by the $N \times N$ Weinstein--Aronszajn matrices 
$W^{(N)}(\lambda)$, whose entries are given by
\formula{
 W_{m,n}(\lambda) & = \langle g_m + \lambda (\A - \lambda \B%%_0
^{\mk{(0)}}
)^{-1} \tilde{\B} g_m, \tilde{\B} g_n \rangle \mk{,}
}
with $0 \le m,n < N$. \mk{In particular, if $W^{(N)}(\lambda)$ is singular, then $\lambda$ is an eigenvalue of the intermediate problem.} More precisely, Aronszajn's theorem states that \mk{(for any $\lambda$)} the multiplicity $\mk{m}^{(N)}(\lambda)$ of an eigenvalue $\lambda$ of the intermediate problem satisfies
\formula{
 \mk{m}^{(N)}(\lambda) & = \mk{m}^{(0)}(\lambda) + \deg \det W^{(N)}(\lambda) ,
}
where $\deg \det W^{(N)}(\lambda)$ denotes the smallest (possibly negative) exponent corresponding to a non-zero term in the Laurent series expansion of $\det W^{(N)}$ around $\lambda$.

Typically, one chooses $g_n$ so that $W_{m,n}(\lambda)$ is easily computed. This is the case when $(\A - \lambda \B%%_0
^{\mk{(0)}}
)^{-1} \tilde{\B} g_n$ is a linear combination of the eigenfunctions $f_m$ of the base problem, that is, $\tilde{\B} g_n$ is a linear combination of $\B%%_0
^{\mk{(0)}}
 f_m$.

In our case $\B%%_0
^{\mk{(0)}}
 f_m = P_m$ and $\tilde{\B} g = (\omega^{-1} - 1) g$, so it is convenient to choose $g_n$ so that $(\omega^{-1} - 1) g_n$ are linear combinations of $P_m$. We cannot take simply $g_n = (\omega^{-1} - 1)^{-1} P_n$ due to a singularity at $0$. To cancel out this singularity, we let
\formula{
% g_n & = \frac{|\cdot|^2 P_n}{\omega^{-1} - 1} = \frac{-\alpha_n P_{n+1} + \beta_n P_n - \gamma_n P_{n-1}}{\omega^{-1} - 1} \, .
 g_n & = \frac{P_n - P_{n+1}}{\omega^{-1} - 1} \, .
}
It follows that
\formula{
 \lambda (\A - \lambda \B%%_0
^{\mk{(0)}}
)^{-1} \tilde{\B} g_n & = \frac{\lambda}{\mu_n - \lambda} \, \omega P_n - \frac{\lambda}{\mu_{n+1} - \lambda} \, \omega P_{n+1} ,
}
and therefore
\formula{
 W_{m,n}(\lambda) & = \int_D \frac{(P_m(x) - P_{m+1}(x)) (P_n(x) - P_{n+1}(x))}{(1 - |x|^2)^{-\alpha/2} - 1} \, dx \\
 & + \frac{\lambda \delta_{m,n} \sigma_n}{\mu_m - \lambda} - \frac{\lambda \delta_{m+1,n} \sigma_n}{\mu_{m+1} - \lambda} - \frac{\lambda \delta_{m,n+1} \sigma_{n+1}}{\mu_m - \lambda} + \frac{\lambda \delta_{m+1,n+1} \sigma_{n+1}}{\mu_{m+1} - \lambda} \, .
% & + \frac{\alpha_m \alpha_n \delta_{m+1,n+1} \sigma_{n+1}}{\mu_{m+1} - \lambda} - \frac{\alpha_m \beta_n \delta_{m+1,n} \sigma_n}{\mu_{m+1} - \lambda} + \frac{\alpha_m \gamma_n \delta_{m+1,n-1} \sigma_{n-1}}{\mu_{m+1} - \lambda} \\
% & - \frac{\beta_m \alpha_n \delta_{m,n+1} \sigma_{n+1}}{\mu_m - \lambda} + \frac{\beta_m \beta_n \delta_{m,n} \sigma_n}{\mu_m - \lambda} - \frac{\beta_m \gamma_n \delta_{m,n-1} \sigma_{n-1}}{\mu_m - \lambda} \\
% & + \frac{\gamma_m \alpha_n \delta_{m-1,n+1} \sigma_{n+1}}{\mu_{m-1} - \lambda} - \frac{\gamma_m \beta_n \delta_{m-1,n} \sigma_n}{\mu_{m-1} - \lambda} + \frac{\gamma_m \gamma_n \delta_{m-1,n-1} \sigma_{n-1}}{\mu_{m-1} - \lambda} .
}
The eigenvalues of the base problem are given by $\mu_n$. It is easily proved that
\formula{
 w^{(N)}(\lambda) & = \expr{\prod_{n = 0}^N (\mu_n - \lambda)} \det W^{(N)}(\lambda)
}
is a polynomial of degree $N + 1$ in $\lambda$. Hence, $\mk{\lambdal_{d,n}^{(N)}}$, $n \ge 0$, is a sequence that consists of the $N + 1$ zeroes of $w^{(N)}(\lambda)$ and all $\mu_n$ for $n > N$, arranged in a non-decreasing order. This proves the other part of Theorem~\ref{th:eigenvalues}.

%
%                            ---------- o ----------
%

\section{Analytical bounds}
\label{sec:analytic}

In this final part we apply Theorem~\ref{th:eigenvalues} to find analytical bounds for the two smallest eigenvalues that correspond to radial eigenfunctions. These bounds are then used to prove Theorem~\ref{th:antisymmetric}.

Recall that $\lambda_{d,0}$ increases with the dimension $d$. By Proposition~\ref{prop:bochner}, the second smallest eigenvalue is equal to either $\lambda_{d+2,0}$ (when the corresponding eigenfunction is antisymmetric) or $\lambda_{d,1}$ (when it is radial). Figure~\ref{fig:gap} suggests that if $d \le 9$ and $\alpha \in (0, 2]$, then the bounds obtained above satisfy $%%\lambdau_0(d + 2, 2)
\mk{\lambdau_{d+2,0}^{(2)}}
 < %%\lambdal_1(d, 1)
\mk{\lambdal_{d,1}^{(1)}}
$, which clearly implies that $\lambda_{d+2,0} < \lambda_{d,1}$. This would extend Theorem~\ref{th:antisymmetric} to $d \le 9$ and $\alpha \in (0, 2]$. However, we could not overcome the technical details unless $\alpha = 1$ or $d \le 2$, the case covered by Theorem~\ref{th:antisymmetric}. We are, however, convinced that at least a computer-assisted proof can be given for $\alpha \in (0, 2]$ and $d \le 9$.

We first consider the general case. The upper bounds of Theorem~\ref{th:eigenvalues} for $N = 2$ are the solutions of the $2 \times 2$ matrix eigenvalue problem 
\formula{
 \begin{pmatrix} \mu_0 \sigma_0 & 0 \\ 0 & \mu_1 \sigma_1 \end{pmatrix} x & = \lambda \begin{pmatrix} \pi_{0,0} & \pi_{0,1} \\ \pi_{1,0} & \pi_{1,1} \end{pmatrix} x ,
}
that is, the solutions of
\formula[eq:upper:eq]{
 (\mu_0 \sigma_0 - \pi_{0,0} \lambda) (\mu_1 \sigma_1 - \pi_{1,1} \lambda) - \pi_{0,1}^2 \lambda^2 = 0
}
(note that $\pi_{0,1} = \pi_{1,0}$). Hence,
\formula[eq:upper:an]{
 %%\lambdau_0(d, 2)
\mk{\lambdau_{d,0}^{(2)}}
 & = \frac{\KK - \sqrt{\KK^2 - 4 \mu_0 \mu_1 \sigma_0 \sigma_1 \LL}}{2 \LL} \, , & %%\lambdau_1(d, 2)
\mk{\lambdau_{d,1}^{(2)}}
 & = \frac{\KK + \sqrt{\KK^2 - 4 \mu_0 \mu_1 \sigma_0 \sigma_1 \LL}}{2 \LL} \, ,
}
with $\KK = \mu_0 \sigma_0 \pi_{1,1} + \mu_1 \sigma_1 \pi_{0,0}$ and $\LL = \pi_{0,0} \pi_{1,1} - \pi_{0,1}^2$. After simplification, we obtain
\formula{
  %%\lambdau_0(d, 2)
\mk{\lambdau_{d,0}^{(2)}}
 & = \frac{\mu_0 \Gamma(\tfrac{\alpha}{2}+2) \Gamma(\tfrac{d}{2} + \alpha +2) (\MM - \sqrt{\MM^2 - 4d(d+\alpha)(\alpha+1)(d+\alpha+4)(d+2\alpha+4)})}{4d\Gamma(\tfrac{d+\alpha}{2}+3) \Gamma(\alpha+2)} \, ,
}
where $\MM = (\alpha+2)(d^2+2\alpha d+4d +2\alpha^2 + 2\alpha)$. By a straightforward, but lengthy calculation we find that
\formula[eq:upper:W]{
  %%\lambdau_0(d, 2)
\mk{\lambdau_{d,0}^{(2)}}
 & \le \frac{\mu_0 \Gamma(\tfrac{\alpha}{2}+2) \Gamma(\tfrac{d}{2} + \alpha +2) (2 d^2 + 4 \alpha d + 8 d - \alpha)}{4d\Gamma(\tfrac{d+\alpha}{2}+3) \Gamma(\alpha+2)} \, . % =: \RR(\alpha,d),
}

The lower bounds can be found analytically for $N = 1$. In this case, $W^{(N)}(\lambda)$ is a $1 \times 1$ matrix with entry
\formula{
 W_{0,0}(\lambda) & = I + \frac{\lambda \sigma_0}{\mu_0 - \lambda} + \frac{\lambda \sigma_1}{\mu_1 - \lambda} \, ,
}
where $I$ is given by~\eqref{eq:int:i}. Therefore, the sequence of the lower bounds $%%\lambdal_n(d, 1)
\mk{\lambdal_{d,n}^{(1)}}
$ consists of the numbers $\mu_n$ for $n \ge 2$ and the two solutions of the equation
\formula[eq:lower:eq]{
 w^{(1)}(\lambda) & = (\mu_0 - \lambda) (\mu_1 - \lambda) I + \lambda \sigma_0 (\mu_1 - \lambda) + \lambda \sigma_1 (\mu_0 - \lambda) = 0 .
}
Note that the above equation with $I$ replaced by $\sigma_0 + \sigma_1$ is a linear equation having solution $\lambda = \mu_0 \mu_1 (\sigma_0 + \sigma_1) / (\mu_0 \sigma_1 + \mu_1 \sigma_0) \in (\mu_0, \mu_1)$. By~\eqref{eq:int:i:est}, $I > \sigma_0 + \sigma_1$, and so one of the solutions of~\eqref{eq:lower:eq} lies between $\mu_0$ and $\mu_1$, while the other one is greater than $\mu_1$.

The two solutions of~\eqref{eq:lower:eq} are easily calculated. It follows that
\formula[eq:lower:an]{
 %%\lambdal_0(d, 1)
\mk{\lambdal_{d,0}^{(1)}}
 & = \frac{\PP - \sqrt{\PP^2 - 4 \mu_0 \mu_1 I \QQ}}{2 \QQ} \, , & %%\lambdal_1(d, 1)
\mk{\lambdal_{d,1}^{(1)}}
 & = \min \expr{\frac{\PP + \sqrt{\PP^2 - 4 \mu_0 \mu_1 I \QQ}}{2 \QQ}, \mu_2} ,
}
where $\PP = (\mu_0 + \mu_1) I - \mu_0 \sigma_1 - \mu_1 \sigma_0$, $\QQ = I - \sigma_0 - \sigma_1$.

We remark that if $I > \sigma_0 + \sigma_1$ is taken as a parameter in the equation~\eqref{eq:lower:eq}, the solutions of this equation decrease as $I$ increases. Since $I$ is given as a series (or in an integral form, see~\eqref{eq:int:i} and~\eqref{eq:int:i2}), we may replace it by a more tractable greater number, given for example by~\eqref{eq:zeta} or~\eqref{eq:J}, and thus obtain lower bounds for the eigenvalues that are slightly weaker, but are expressed in closed form. This will help in studying the case $d \le 2$. When $\alpha = 1$, however, the integral in the definition $I$ can be expressed in closed form.% and so there is no need to use~\eqref{eq:zeta}.

\subsection{Estimates for $d\le 9$ and $\alpha=1$}

In this case
\formula{
 \mu_0 & = \frac{2 \Gamma(\tfrac{3}{2}) \Gamma(\tfrac{d + 1}{2})}{\Gamma(\tfrac{d}{2})} = \frac{\sqrt{\pi} \, \Gamma(\tfrac{d + 1}{2})}{\Gamma(\tfrac{d}{2})} \, , &
 \mu_1 & = \frac{2 \Gamma(\tfrac{5}{2}) \Gamma(\tfrac{d + 3}{2})}{\Gamma(\tfrac{d + 2}{2})} = \frac{\tfrac{3}{2} \sqrt{\pi} \, \Gamma(\tfrac{d + 3}{2})}{\Gamma(\tfrac{d + 2}{2})} \, , \\
 \sigma_0 & = \frac{\pi^{d/2} \Gamma(\tfrac{d}{2}) \Gamma(\tfrac{3}{2})}{\tfrac{d + 1}{2} \Gamma(\tfrac{d}{2}) \Gamma(\tfrac{d + 1}{2})} = \frac{\pi^{(d + 1)/2}}{2 \Gamma(\tfrac{d + 3}{2})} \, , &
 \sigma_1 & = \frac{\pi^{d/2} \Gamma(\tfrac{d}{2}) \Gamma(\tfrac{5}{2})}{\tfrac{d + 5}{2} \Gamma(\tfrac{d + 2}{2}) \Gamma(\tfrac{d + 3}{2})} = \frac{3 \pi^{(d + 1)/2}}{d (d + 5) \Gamma(\tfrac{d + 3}{2})} \, ,
}
and
\formula{
\begin{gathered}
\begin{aligned}
 \pi_{0,0} & = \frac{\pi^{d/2}}{\Gamma(\tfrac{d}{2} + 2)} \, , & \qquad
 \pi_{0,1} & = \pi_{1,0} = \frac{\pi^{d/2} \Gamma(\tfrac{3}{2})}{\Gamma(\tfrac{1}{2}) \Gamma(\tfrac{d}{2} + 3)} = \frac{\pi^{d/2}}{2 \Gamma(\tfrac{d}{2} + 3)} \, ,
\end{aligned} \\
 \pi_{1,1} = \frac{\pi^{d/2} \Gamma(\tfrac{d}{2}) \Gamma(\tfrac{d}{2} + 2) (\Gamma(\tfrac{5}{2}))^2}{(\Gamma(\tfrac{3}{2}) \Gamma(\tfrac{d}{2} + 1))^2 \Gamma(\tfrac{d}{2} + 4)} = \frac{9 \pi^{d/2} (d + 2)}{4 d \, \Gamma(\tfrac{d}{2} + 4)} \, .
\end{gathered}
}
Consequently, the equation~\eqref{eq:upper:eq}, whose solutions are the upper bounds for eigenvalues, takes the form (after multiplication of both sides by $\tfrac{1}{9} \pi^{-d} (d \, \Gamma(\tfrac{d}{2}))^2$)
\formula[eq:upper:eq1]{
 \expr{\frac{\pi}{d + 1} - \frac{4 \lambda}{d (d + 2)}} \expr{\frac{\pi}{d + 5} - \frac{4 \lambda}{(d + 4) (d + 6)}} - \frac{16 \lambda^2}{9 (d + 2)^2 (d + 4)^2} = 0 ,
}
and finally
\formula[eq:upper:an1]{
 %%\lambdau_0(d, 2)
\mk{\lambdau_{d,0}^{(2)}}
 & = \frac{3 \pi (d + 2) (d + 4) (3 (d^2 + 6 d + 4) - \sqrt{(d^2 + 6 d + 16)^2 - 112})}{32 (d + 1) (d + 3) (d + 5)} \, , \\
 %%\lambdau_1(d, 2)
\mk{\lambdau_{d,1}^{(2)}}
 & = \frac{3 \pi (d + 2) (d + 4) (3 (d^2 + 6 d + 4) + \sqrt{(d^2 + 6 d + 16)^2 - 112})}{32 (d + 1) (d + 3) (d + 5)} \, .
}
On the other hand, by~\eqref{eq:int:i2},
\formula{
 I & = \frac{(d + 3)^2 \pi^{d/2}}{d^2 \Gamma(\tfrac{d}{2})} \int_0^1 \frac{\sqrt{t} (1 - t)^{d/2 + 1}}{1 - \sqrt{t}} \, dt \\
 & = \frac{(d + 3)^2 \pi^{d/2}}{d^2 \Gamma(\tfrac{d}{2})} \int_0^1 (\sqrt{t} + t) (1 - t)^{d/2} dt \\
 & = \frac{(d + 3)^2 \pi^{d/2}}{d^2} \expr{\frac{4}{(d + 2) (d + 4) \Gamma(\tfrac{d}{2})} + \frac{\sqrt{\pi} \, d}{(d + 1) (d + 3) \Gamma(\tfrac{d + 1}{2})}} \, .
}
Therefore, the equation~\eqref{eq:lower:eq} for the lower bounds for eigenvalues simplifies to (after multiplication of both sides by $2 \pi^{-(d + 1)/2} \Gamma(\tfrac{d + 3}{2})$)
\formula[eq:lower:eq1]{
 & \frac{d + 3}{d} \expr{1 + \frac{4 (d + 1) (d + 3)}{d (d + 2) (d + 4) \BB}} \expr{\frac{\pi}{\BB} - \lambda} \expr{\frac{3 \pi (d + 1)}{2 d \BB} - \lambda} \\
% I think this line:
%  & \hspace*{10em} + \lambda \expr{\frac{\pi}{\BB} - \lambda} + \frac{6 \lambda}{d (d + 5)} \expr{\frac{3 \pi (d + 1)}{2 d \BB} - \lambda} = 0 ,\\
% should be replaced by this:
  & \hspace*{10em} + \lambda \expr{\frac{3 \pi (d + 1)}{2 d \BB} - \lambda} + \frac{6 \lambda}{d (d + 5)} \expr{\frac{\pi}{\BB}  - \lambda} = 0 ,
}
where $\BB = B(\tfrac{d}{2}, \tfrac{1}{2}) = \sqrt{\pi} \, \Gamma(\tfrac{d}{2}) / \Gamma(\tfrac{d + 1}{2})$.

We claim that $%%\lambdau_0(d, 2)
\mk{\lambdau_{d,0}^{(2)}}
 < \tfrac{3 \pi}{16} (d + 3)$ and $%%\lambdal_1(d, 1)
\mk{\lambdal_{d,1}^{(1)}}
 > \min(\tfrac{3 \pi}{16} (d + 5), \mu_2)$ when $d \le 9$ (in fact, this holds for all $d$).

Observe that the coefficient of $\lambda^2$ in the left-hand side of~\eqref{eq:upper:eq1} is positive (this follows easily from the inequality $(d + 2) (d + 4) > d (d + 6)$), and substituting $\lambda = \tfrac{3 \pi}{16} (d + 3)$ in the left-hand side of~\eqref{eq:upper:eq1} gives a negative number:
\formula{
 & \expr{\frac{\pi}{d + 1} - \frac{3 \pi (d + 3)}{4 d (d + 2)}} \expr{\frac{\pi}{d + 5} - \frac{3 \pi (d + 3)}{4 (d + 4) (d + 6)}} - \frac{\pi^2 (d + 3)^2}{16 (d + 2)^2 (d + 4)^2} \\
 & \hspace*{13em} = -\pi^2 \frac{8 d^4 + 96 d^3 + 409 d^2 + 726 d + 459}{2 d (d + 1) (d + 2)^2 (d + 4)^2 (d + 5) (d + 6)} \, .
}
This proves that $%%\lambdau_0(d, 2)
\mk{\lambdau_{d,0}^{(2)}}
 < \tfrac{3 \pi}{16} (d + 3)$. 

In a similar manner, the coefficient of $\lambda^2$ in the left-hand side of~\eqref{eq:lower:eq1} is positive (because $\BB > 0$ and $\tfrac{d + 3}{d} > 1 + \tfrac{6}{d (d + 5)}$), and substituting $\lambda = \tfrac{3 \pi}{16} (d + 5)$ in the left-hand side of~\eqref{eq:lower:eq1} gives a negative number. Indeed, %by the fact that \MK{!!!$\tfrac{5}{2} d^{-1/2} \le \BB \le \tfrac{8}{3} \, d^{-1/2}$!!!}, one obtains
\formula{
 & \frac{d + 3}{d} \expr{1 + \frac{4 (d + 1) (d + 3)}{d (d + 2) (d + 4) \BB}} \expr{\frac{\pi}{\BB} - \frac{3 \pi (d + 5)}{16}} \expr{\frac{3 \pi (d + 1)}{2 d \BB} - \frac{3 \pi (d + 5)}{16}} \\
% & \hspace*{3em} + \frac{3 \pi (d + 5)}{16} \expr{\frac{\pi}{\BB} - \frac{3 \pi (d + 5)}{16}} + \frac{9 \pi}{8 d} \expr{\frac{3 \pi (d + 1)}{2 d \BB} - \frac{3 \pi (d + 5)}{16}} \\
 & \hspace*{3em} + \frac{3 \pi (d + 5)}{16} \expr{\MK{\frac{3 \pi (d + 1)}{2 d \BB}} - \frac{3 \pi (d + 5)}{16}} + \frac{9 \pi}{8 d} \expr{\MK{\frac{\pi}{\BB}} - \frac{3 \pi (d + 5)}{16}} \\
 & = -\frac{\pi^2 (d + 3)}{256 d^3 (d + 2) (d + 4) \BB^3} \, \bigl(\MK{12} %36 
d^5 \BB^2 + (96 d^4 \BB - 27 d^5 \BB^3) + (-1536 d^3 + \MK{240} %504 
d^4 \BB^2) \\
 & \hspace*{3em} + (1920 d^3 \BB - 297 d^4 \BB^3) + (-7680 d^2 + \MK{1032} %1800 
d^3 \BB^2) + (8256 d^2 \BB - 1026 d^3 \BB^3) \\
 & \hspace*{3em} + (-10752 d + \MK{1128} %1224 
d^2 \BB^2) + (10752 d \BB - 1080 d^2 \BB^3) + (-4608 + \MK{180} %- 972
 d \BB^2) + 4320 \BB \bigr) , %\\
% & \le -\frac{\pi^2 (d + 3)}{256 d^3 (d + 2) (d + 4) B^3} \, \bigl( \MK{75} %225 
%d^4 - 272 d^{7/2} \MK{- 36} %+ 1614 
%d^3 - 832 d^{5/2} \\
% & \hspace*{3em} \MK{- 1230} %+ 3570 
%d^2 + 1184 d^{3/2} - \MK{3702} %3102 
%d + 6400 d^{1/2} - \MK{3483} %11520 
%+ 10800 d^{-1/2}\bigr) ,
}
\MK{and it is elementary, but rather tedious, to verify that the right-hand side is negative for $d \le 9$ (with some additional work it is not very difficult to extend this statement to general $d \ge 1$).} Therefore, $%%\lambdal_1(d, 1)
\mk{\lambdal_{d,1}^{(1)}}
 > \min(\tfrac{3 \pi}{16} (d + 5), \mu_2)$, and our claim is proved.

Observe that until now we did not use the restriction $d \le 9$ in an essential way. This condition is needed only for the final step: we have
\formula{
 \mu_2 & = \frac{\Gamma(\tfrac{7}{2}) \Gamma(\tfrac{d + 5}{2})}{\Gamma(\tfrac{d + 4}{2})} = \frac{15 \sqrt{\pi} \, \Gamma(\tfrac{d + 5}{2})}{8 \, \Gamma(\tfrac{d + 4}{2})} \, ,
}
and by inspection, $\mu_2 > \tfrac{3 \pi}{16} (d + 5)$ for $d \le 9$ (and this inequality is \emph{not} true for $d \ge 10$). It follows that $%%\lambdal_1(d, 1)
\mk{\lambdal_{d,1}^{(1)}}
 > \tfrac{3 \pi}{16} (d + 5)$ for $d \le 9$ and $%%\lambdau_0(d, 2)
\mk{\lambdau_{d,0}^{(2)}}
 < \tfrac{3 \pi}{16} (d + 3)$ for all $d$. In particular, $%%\lambdau_0(d + 2, 2)
\mk{\lambdau_{d+2,0}^{(2)}}
 < %%\lambdal_1(d, 1)
\mk{\lambdal_{d,1}^{(1)}}
$ when $d \le 9$, as desired. This proves half of Theorem~\ref{th:antisymmetric}.

We remark that using a similar approach, with more careful estimates, one can extend the above result to $\alpha = 1$ and $d = 10$, see Figure~\ref{fig:gap}.

\subsection{Estimate for $d \le 2$ and $\alpha \in (0,2]$}
\label{subsec:d2}

We start with two technical lemmas.

\begin{lemma}\label{lem:4gamma}
The function
\formula{
h(\alpha)=\frac{ \Gamma(\alpha+3)\Gamma(\frac{\alpha}{2} + \frac{9}{2}) }{\Gamma(\frac{\alpha}{2}+2)\Gamma(\alpha+\frac{9}{2})}
}
is increasing on $[0,2]$.
\end{lemma}
\begin{proof}
Using Legendre duplication formula for the gamma function, we check that 
\formula{
 h(2z) = 2^{-3/2} \, \frac{\Gamma(z+\tfrac{9}{2})\Gamma(z+\tfrac{3}{2})}{\Gamma(z+\tfrac{9}{4})\Gamma(z+\tfrac{11}{4})}.
}
The logarithmic derivative of $H(z) = h(2 z)$ is given in terms of the digamma function $\psi$,
\begin{equation}\label{lemma5_proof1}
\frac{H'(z)}{H(z)}=\psi(z+\tfrac{9}{2})+\psi(z+\tfrac{3}{2})-\psi(z+\tfrac{9}{4})-\psi(z+\tfrac{11}{4}). 
\end{equation}
The proof will be complete if we show that this quantity is positive for all $z \in (0, 1)$.

Let us define $f_a(z)=\psi(z+a)-\psi(z)-a/z$.  Using the functional equation $\psi(z+1)=\psi(z)+\tfrac{1}{z}$
we rewrite  \eqref{lemma5_proof1} in the following equivalent form
\formula[lemma5_proof2]{
\frac{H'(z)}{H(z)} & = \expr{\frac{1}{z+\tfrac{7}{2}}+\frac{\tfrac{1}{4}}{z+\tfrac{9}{4}}+\frac{\tfrac{3}{4}}{z+\tfrac{11}{4}}-\frac{1}{z+\tfrac{3}{2}}}
+f_{3/4}(z+\tfrac{11}{4})+f_{1/4}(z+\tfrac{9}{4}). 
}
According to Theorem~2.1 in~\cite{bib:m10},
the function 
$z\mapsto f_a(z)$ is completely monotone for $a\in (0,1)$. In particular, $f_a(z)>0$ for all $z \in (0,\infty)$ and $a \in (0,1)$. 
Using this result, formula~\eqref{lemma5_proof2} and the %following 
identity
\formula{
\frac{1}{z+\tfrac{7}{2}}+\frac{\tfrac{1}{4}}{z+\tfrac{9}{4}}+\frac{\tfrac{3}{4}}{z+\tfrac{11}{4}}-\frac{1}{z+\tfrac{3}{2}} & =
%\frac{1}{z+7/2}+\frac{1/4}{z+9/4}+\frac{3/4}{z+11/4}-\frac{1}{z+3/2}=
\frac{(32z^3 + 172z^2 + 228z + 3)}{32(z+\tfrac{7}{2})(z+\tfrac{9}{4})(z+\tfrac{11}{4})(z+\tfrac{3}{2})} \, , 
}
we prove that $H'(z) / H(z) > 0$ %the quantity in \eqref{lemma5_proof2} is strictly positive 
for all $z>0$, as desired.%thus $h(\alpha)$ is an increasing function. 
\end{proof}

\begin{lemma}\label{lem:4gamma2}
The function
\formula{
 h(\alpha) & = \frac{\Gamma(\frac{\alpha}{2}+2)\Gamma(\alpha+\frac{7}{2})}{\Gamma(\frac{\alpha}{2}+\frac{9}{2}) \Gamma(\alpha+1)}
}
is increasing on $[0, 2]$.
\end{lemma}

\begin{proof}
This result is much simpler than the previous one: by Theorem~1.2.5 in~\cite{bib:aar99}, the logarithmic derivative of $h$ equals
\formula{
 \frac{h'(\alpha)}{h(\alpha)} & = -\sum_{k=0}^\infty \left( 
\frac{\frac12}{\frac{\alpha}{2}+2 + k} + \frac{1}{\alpha+\frac{7}{2}+k} - \frac{\frac12}{\frac{\alpha}{2}+\frac{9}{2} +k} - \frac{1}{\alpha+1+k} \right) \\
&= -\sum_{k=0}^\infty \left( \frac{5}{(2k+\alpha+4)(2k+\alpha+9)} - \frac{5}{(k+\alpha+1)(2k+2\alpha+7)} \right) > 0 ,
}
and hence $h$ increasing.
\end{proof}

Denote the upper bound~\eqref{eq:upper:W} for $%%\lambdau_0(d, 2)
\mk{\lambdau_{d,0}^{(2)}}
$ by $\RR(d)$. We begin \mk{by} checking that $\mu_2 > \RR(d+2)$ for $d=1$ and $d=2$. For $d=1$ this inequality is equivalent to
\formula{
 2(11\alpha+42)\Gamma(\tfrac{\alpha}{2}+2)\Gamma(\alpha+\tfrac{7}{2}) <
 \Gamma(\alpha+5)\Gamma(\tfrac{\alpha}{2} + \tfrac{9}{2}),
}
and since $11\alpha+42 < 12\alpha+42$, it suffices to show
\formula[eq:tmp1]{
24 & < (\alpha+3)(\alpha+4) \, \frac{ \Gamma(\alpha+3)\Gamma(\tfrac{\alpha}{2} + \tfrac{9}{2}) }{\Gamma(\tfrac{\alpha}{2}+2)\Gamma(\alpha+\tfrac{9}{2})}.
}
Both sides of the above inequality are equal for $\alpha=0$, and, by Lemma~\ref{lem:4gamma}, the right hand side is increasing in $\alpha$. Inequality~\eqref{eq:tmp1} follows.

For $d=2$ the inequality $\mu_2 > \RR(d+2)$ takes the form
\formula{
\frac{2^\alpha \Gamma(\tfrac{\alpha}{2}+1) (\Gamma(\tfrac{\alpha}{2}+1))^2 \Gamma(\alpha+4)(15\alpha+64)}{16 \, \Gamma(\tfrac{\alpha}{2}+5) \Gamma(\alpha+2)} & < 2^{\alpha-2} (\Gamma(\tfrac{\alpha}{2}+3))^2,
}
which, after simplification, is equivalent to the elementary inequality
\formula{
 (\alpha+3)(15\alpha+64) < 2(\tfrac{\alpha}{2}+2)^3(\tfrac{\alpha}{2}+3)(\tfrac{\alpha}{2}+4) .
}

We recall that if we replace in equation~\eqref{eq:lower:eq} the number $I$ by a larger number $J$, which we choose to be the right-hand side of~\eqref{eq:J}, then the larger root of this equation is less than $%%\lambdal_1(d, 1)
\mk{\lambdal_{d,1}^{(1)}}
$. Hence, in order to prove that $%%\lambdau_0(d+2,2)
\mk{\lambdau_{d+2,0}^{(2)}}
 < %%\lambdal_1(d, 1)
\mk{\lambdal_{d,1}^{(1)}}
$, it suffices to show that $F(\RR(d+2)) < 0$, where
\formula{
F(\lambda) &=  (\mu_0 - \lambda) (\mu_1 - \lambda) J + \lambda \sigma_0 (\mu_1 - \lambda) + \lambda \sigma_1 (\mu_0 - \lambda) \\
&= (J-\sigma_0-\sigma_1)\lambda^2 - ((\mu_0+\mu_1)J - \sigma_0 \mu_1 - \sigma_1 \mu_0) \lambda + \mu_0 \mu_1 J.
}
Direct calculation gives
\formula[eq:d2:q]{
 F(\RR(d+2)) & = \frac{(d+\alpha)(d+\alpha+2)\mu_0^2 \sigma_0}{d^2\alpha} \, g(T),
}
%\formula[eq:d2:q]{
%F(\RR(d+2)) &= 
%C(aA^2 + bA + (2+\alpha)) = Cg(A),
%}
where $g(t) = at^2+bt+\alpha+2$, and
\formula{
 a & = \frac{(8+2d-\alpha d)(d+\alpha)}{16 d (d+2)^2 (d+\alpha+4)} \, , \\
 b & = -\frac{16\alpha +12\alpha^2+2\alpha^3+32d +16\alpha d -\alpha^3 d +8d^2 -\alpha^2 d^2}{8d(d+2)(d+\alpha+4)} \, , \\
 T & = \frac{(2(d+2)^2+(4\alpha+8)(d+2)-\alpha)  \Gamma(\tfrac{\alpha}{2}+2)\Gamma(\tfrac{d}{2} + \alpha+3)}
{\Gamma(\tfrac{d+\alpha}{2}+4)\Gamma(\alpha+2)} \, .
}
The proof of the other half of Theorem~\ref{th:antisymmetric} will be complete once we show that $g(T) < 0$ for $d = 1$ and $d = 2$.

We consider $d=1$ first. In this case
\formula{
 a & = \frac{(10-\alpha)(\alpha+1)}{144(\alpha+5)} \, ,\\
 b & = -\frac{\alpha^3+11\alpha^2+32\alpha+40}{24(\alpha+5)} \, ,\\
 T & = \frac{(42+11\alpha)\Gamma(\frac{\alpha}{2}+2)\Gamma(\alpha+\frac{7}{2})}{\Gamma(\frac{\alpha}{2}+\frac{9}{2}) \Gamma(\alpha+2)} \, .
}
We will show that
\formula[eq:d1]{
 \frac{2 \alpha + 12}{\alpha+1} < T < 6\alpha + 12 .
}
%note: one could also prove $A>-b/(2a)$
Knowing this, in order to prove that $g(T) < 0$ it suffices to show that $g(\frac{2 \alpha + 12}{\alpha+1})<0$ and $g(6\alpha + 12)<0$. This can be done by a~direct calculation:
\formula{
 g(\tfrac{2 \alpha+12}{\alpha+1}) &= -\frac{\alpha^2(3\alpha^2+16\alpha+8)}{36(\alpha+1)(\alpha+5)} < 0,\\
 g(6\alpha + 12) &= - \frac{\alpha^2 (\alpha+2)^2}{2(\alpha+5)} < 0.
}
We come back to~\eqref{eq:d1}. We obtain
\formula{
T > \frac{42+7\alpha}{\alpha+1}\,\, \frac{\Gamma(\frac{\alpha}{2}+2)\Gamma(\alpha+\frac{7}{2})}{\Gamma(\frac{\alpha}{2}+\frac{9}{2}) \Gamma(\alpha+1)} \, .
}
By Lemma~\ref{lem:4gamma2}, the ratio of gamma functions in the right-hand side is increasing, and its value for $\alpha = 0$ is equal to $\tfrac{2}{7}$. The lower bound in~\eqref{eq:d1} follows. To prove the upper bound, we write
\formula{
T <  \frac{(12\alpha + 42)\Gamma(\frac{\alpha}{2}+2)\Gamma(\alpha+\frac{7}{2})}{\Gamma(\frac{\alpha}{2}+\frac{9}{2}) \Gamma(\alpha+2)}
= (6\alpha + 12) \, \frac{\Gamma(\frac{\alpha}{2}+1)\Gamma(\alpha+\frac{9}{2})}{\Gamma(\frac{\alpha}{2}+\frac{9}{2}) \Gamma(\alpha+2)}.
}
By Lemma~\ref{lem:4gamma}, the ratio of four gamma functions is decreasing, and its value at~$0$ equals~$1$. This completes the proof of~\eqref{eq:d1}.

The case $d=2$ requires less effort. We have
\formula{
a&=\frac{(6-\alpha)(\alpha+2)}{256(\alpha+6)} \, ,\\
b&=-\frac{\alpha^2+6\alpha+12}{8(\alpha+6)} \, ,\\
T&=\frac{8(15\alpha+64)(\alpha+2)(\alpha+3)}{(\alpha+4)(\alpha+6)(\alpha+8)} \, ,
}
%Using computer algebra system one can calculate
%\formula{
%g(A)&=-\frac{\alpha^2(\alpha+2)(281 \alpha^5  + 4528 \alpha^4  + 28281 \alpha^3  + 84510 \alpha^2  + 117852 \alpha + 57768)}{4(\alpha+4)^2(\alpha+6)^3(\alpha+8)^2} < 0,
%}
%as desired, however, this is difficult to verify by hand.
%Therefore we prefer to provide an alternative proof. It is easy to check that
and it is easy to check that
\formula{
16 < T < \frac{128(\alpha+2)(\alpha+3)}{(\alpha+6)(\alpha+8)} < \frac{32(\alpha+2)(\alpha+6)}{3(\alpha+8)} \, .
}
Furthermore, by a direct calculation,
\formula{
 g(16) &= \frac{-2\alpha^2}{\alpha+6} < 0, \\
 g\!\expr{\frac{32(\alpha+2)(\alpha+6)}{3(\alpha+8)}} &= - \frac{\alpha^2(\alpha+2)(4\alpha^2+28\alpha+31)}{9(\alpha+8)^2} < 0,
}
and, consequently, $g(T)<0$. This completes the proof of Theorem~\ref{th:antisymmetric}.

We note that apparently the expression in~\eqref{eq:d2:q} is negative also for $d=3$ (which would extend Theorem~\ref{th:antisymmetric} to this case), but we are unable to prove it rigorously. \MK{For $d \ge 4$, more refined estimates are needed.}

\bigskip

\noindent
\mk{\textbf{Acknowledgements.} The authors thank the anonymous referees for pointing out several mistakes in the initial version of this article and for providing numerous helpful comments.}

%
%                            ---------- o ----------
%

%
%                            ---------- o ----------
%

\clearpage

\section*{Tables and figures}
\addcontentsline{toc}{toc}{Tables and figures}

\begin{figure}[h!]
\centering
\begin{tabular}{cc}
\includegraphics[width=0.45\textwidth]{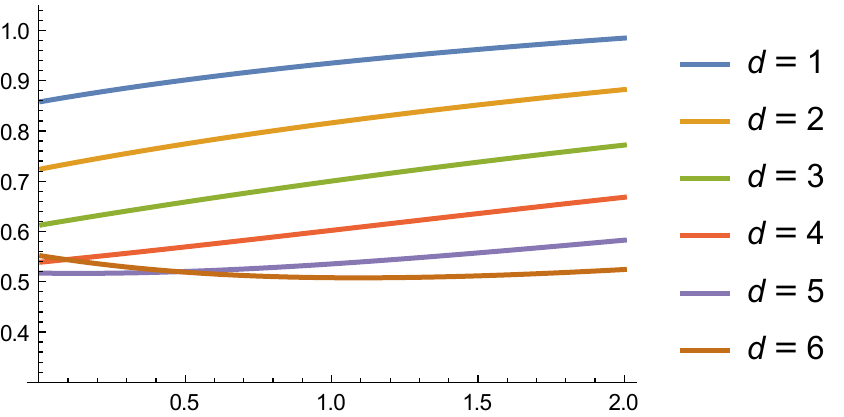} &
\includegraphics[width=0.45\textwidth]{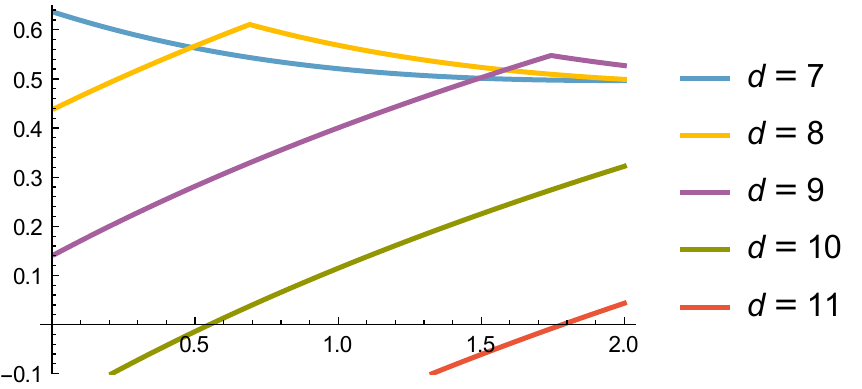}
\\
{\footnotesize $d = 1, 2, \ldots, 6$} &
{\footnotesize $d = 7, 8, \ldots, 11$}
\end{tabular}
\caption{Plot of the gap between the upper bound for $(\lambda_{d+2,0})^{1/\alpha}$ and the lower bound for $(\lambda_{d,1})^{1/\alpha}$, namely of $(%%\lambdal_1(d, 1)
\mk{\lambdal_{d,1}^{(1)}}
)^{1/\alpha} - (%%\lambdau_0(d + 2, 2)
\mk{\lambdau_{d+2,0}^{(2)}}
)^{1/\alpha}$, for $\alpha \in (0, 2]$ \mk{and $d = 1, 2, \ldots, 11$}. %%Different curves correspond to $d = 1, 2, \dots, 10$. 
Note different behaviour for $d \le 7$ \mk{(in which case the lower bound $%%\lambdal_1(d, 1)
\mk{\lambdal_{d,1}^{(1)}}
$ is equal to the larger root of the Weinstein--Aronszajn determinant $w^{(1)}(\lambda)$)} and $d \ge 10$ \mk{(when $%%\lambdal_1(d, 1)
\mk{\lambdal_{d,1}^{(1)}}
 = \mu_2$)}, with transition clearly visible for $d = 8$ and $d = 9$.}
\label{fig:gap}
\end{figure}

\begin{figure}[ht]
\centering
\begin{tabular}{ccc}
\includegraphics[width=0.3\textwidth]{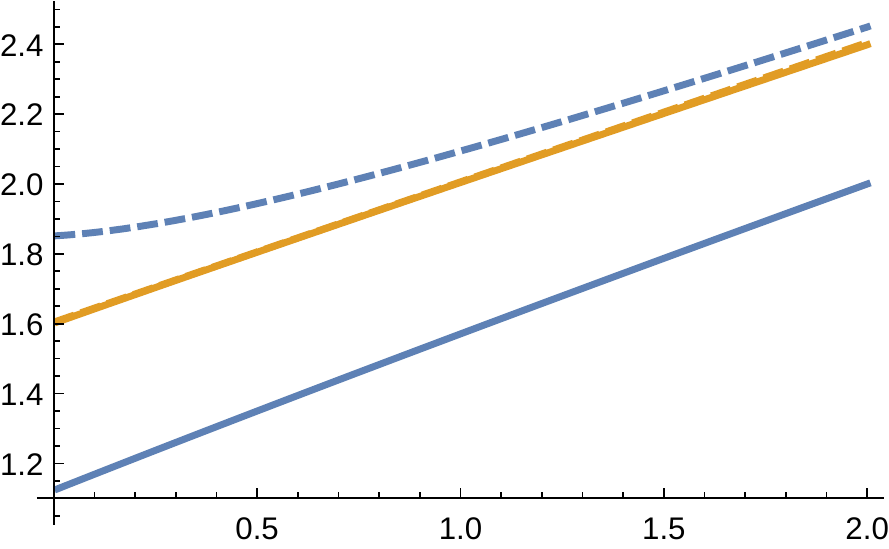} &
\includegraphics[width=0.3\textwidth]{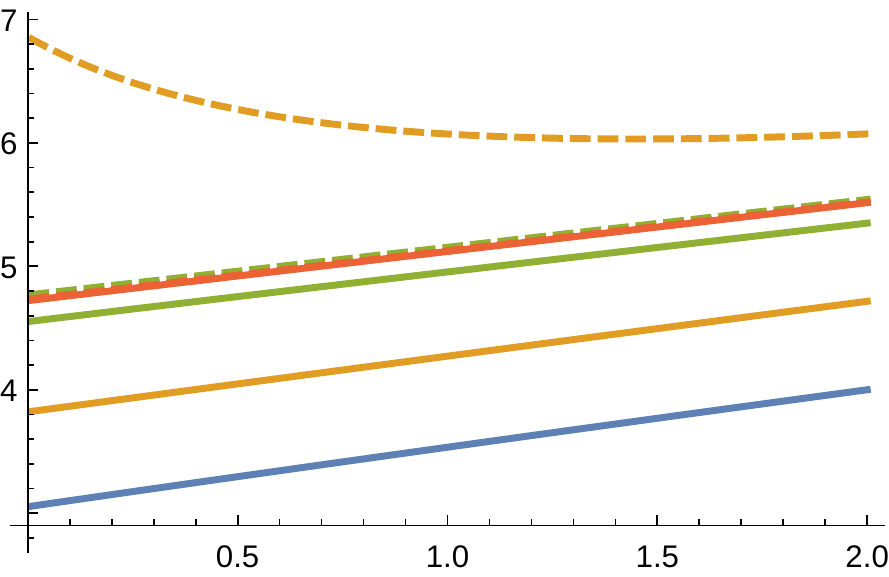} &
\includegraphics[width=0.3\textwidth]{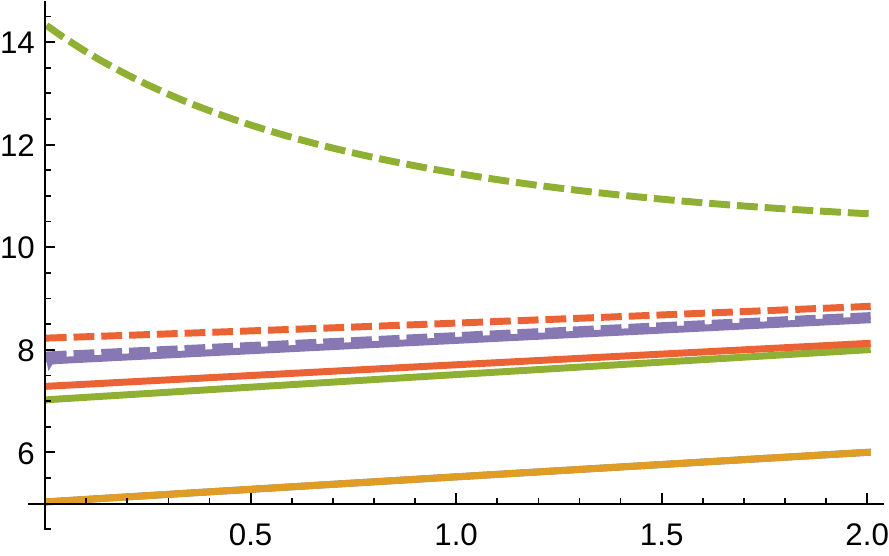}
\\
{\footnotesize $d = 2$, $n = 0$} &
{\footnotesize $d = 2$, $n = 1$} &
{\footnotesize $d = 2$, $n = 2$} \\
\includegraphics[width=0.3\textwidth]{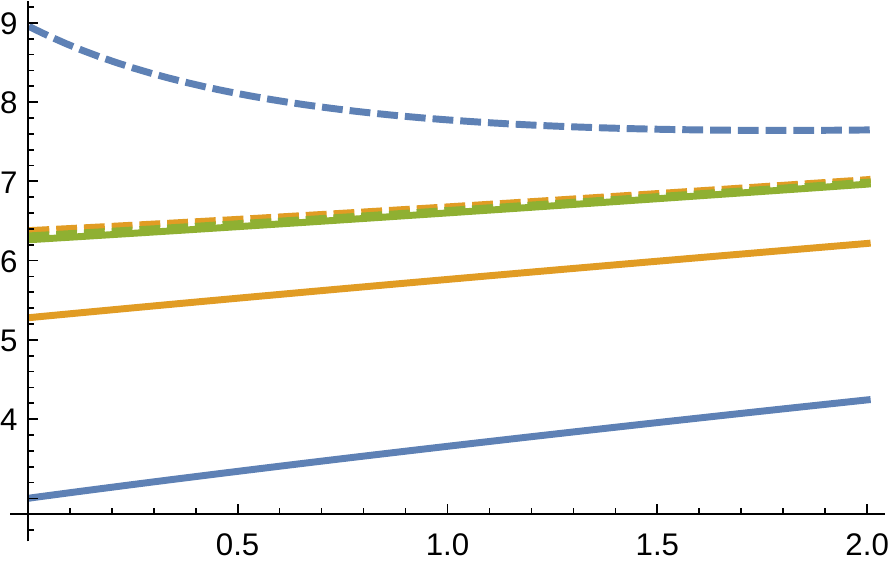} &
\includegraphics[width=0.3\textwidth]{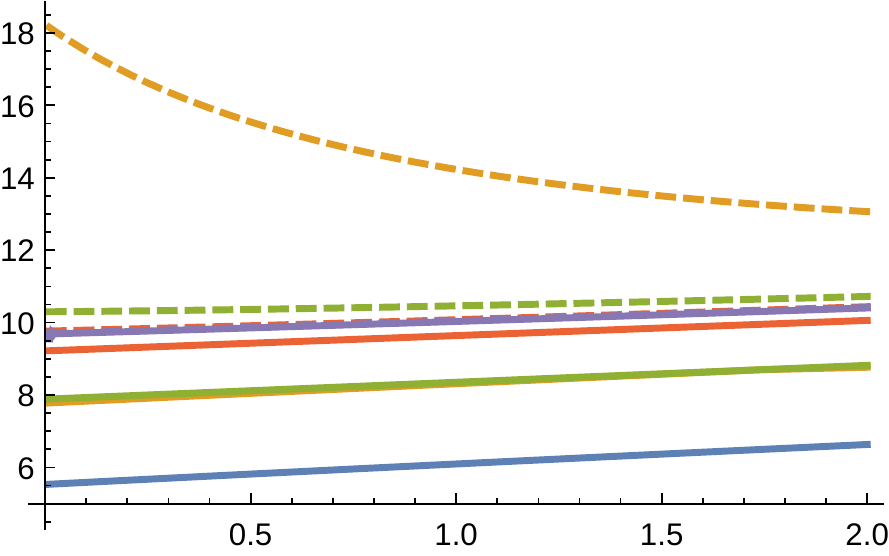} &
\includegraphics[width=0.3\textwidth]{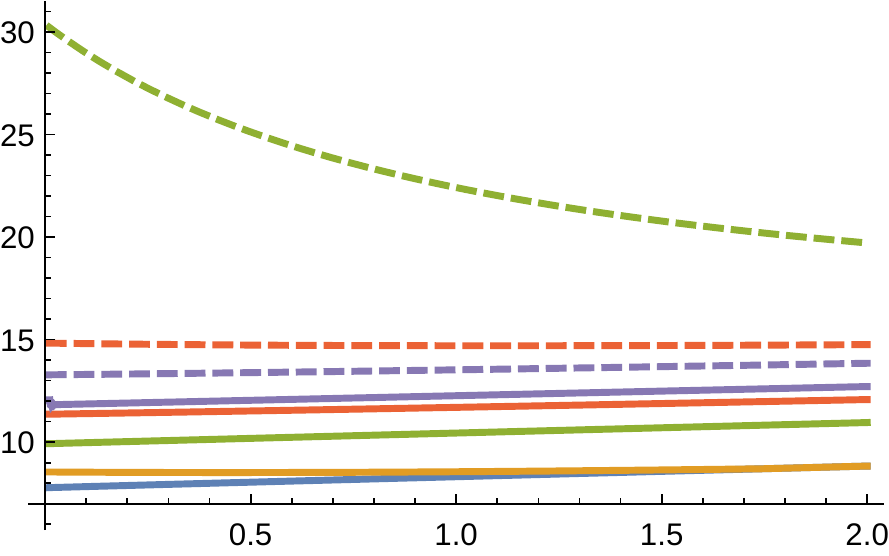}
\\
{\footnotesize $d = 9$, $n = 0$} &
{\footnotesize $d = 9$, $n = 1$} &
{\footnotesize $d = 9$, $n = 2$}
\end{tabular}
\caption{Plots of lower bounds $(%%\lambdal_n(d, N-1)
\mk{\lambdal_{d,n}^{(N-1)}}
)^{1/\alpha}$ (continuous line) and the upper bounds $(\mk{\lambdau_{d,n}^{(N)}})^{1/\alpha}$ (dashed line) for $\alpha \in (0, 2]$, with $N = 1$ (blue), $2$ (yellow), $3$ (green), $4$ (red), $5$ (dark blue).}
\label{fig:eigenvalues}
\end{figure}

\begin{figure}[h!]
\centering
\begin{tabular}{ccc}
\includegraphics[width=0.3\textwidth]{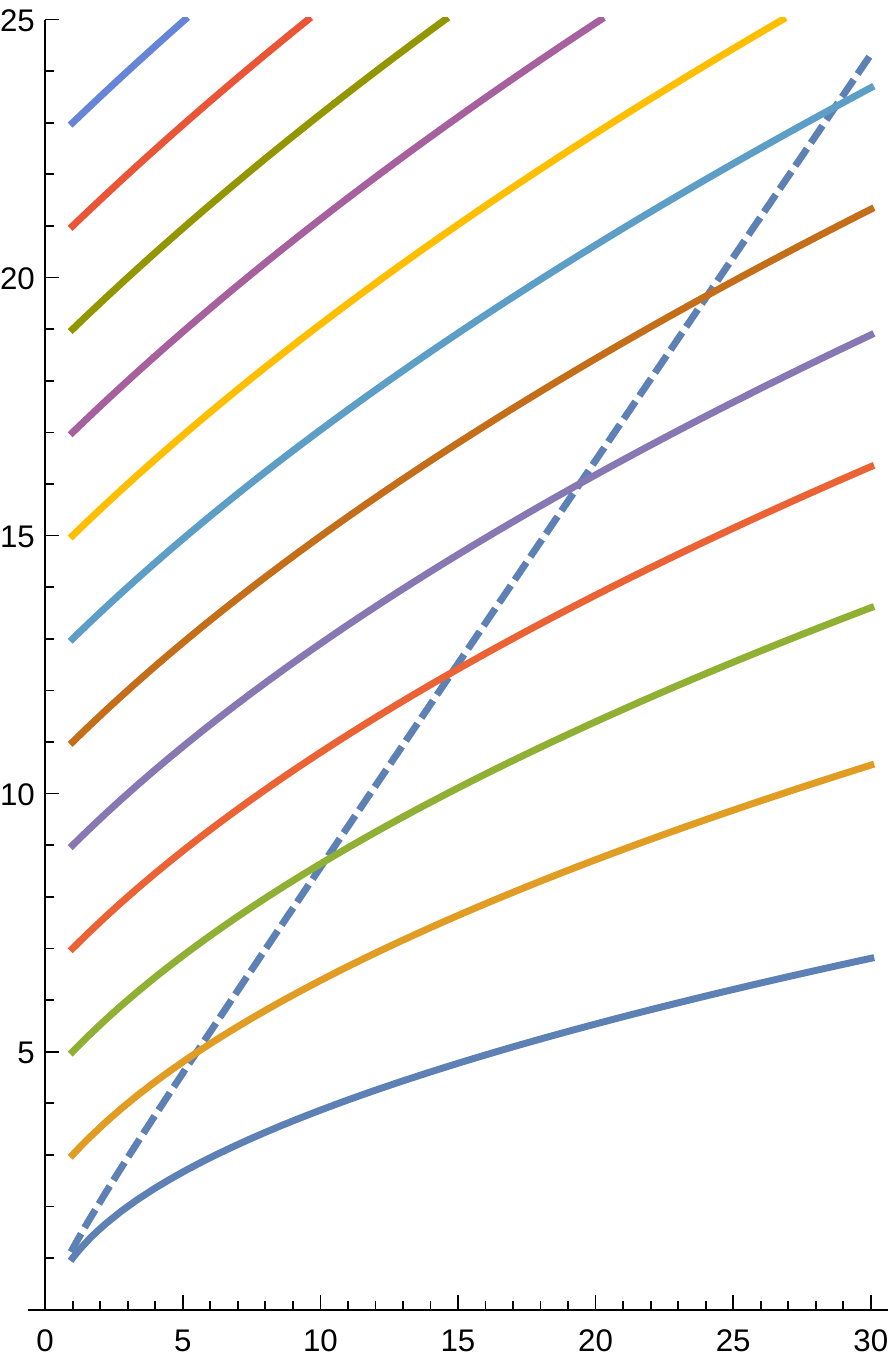} &
\includegraphics[width=0.3\textwidth]{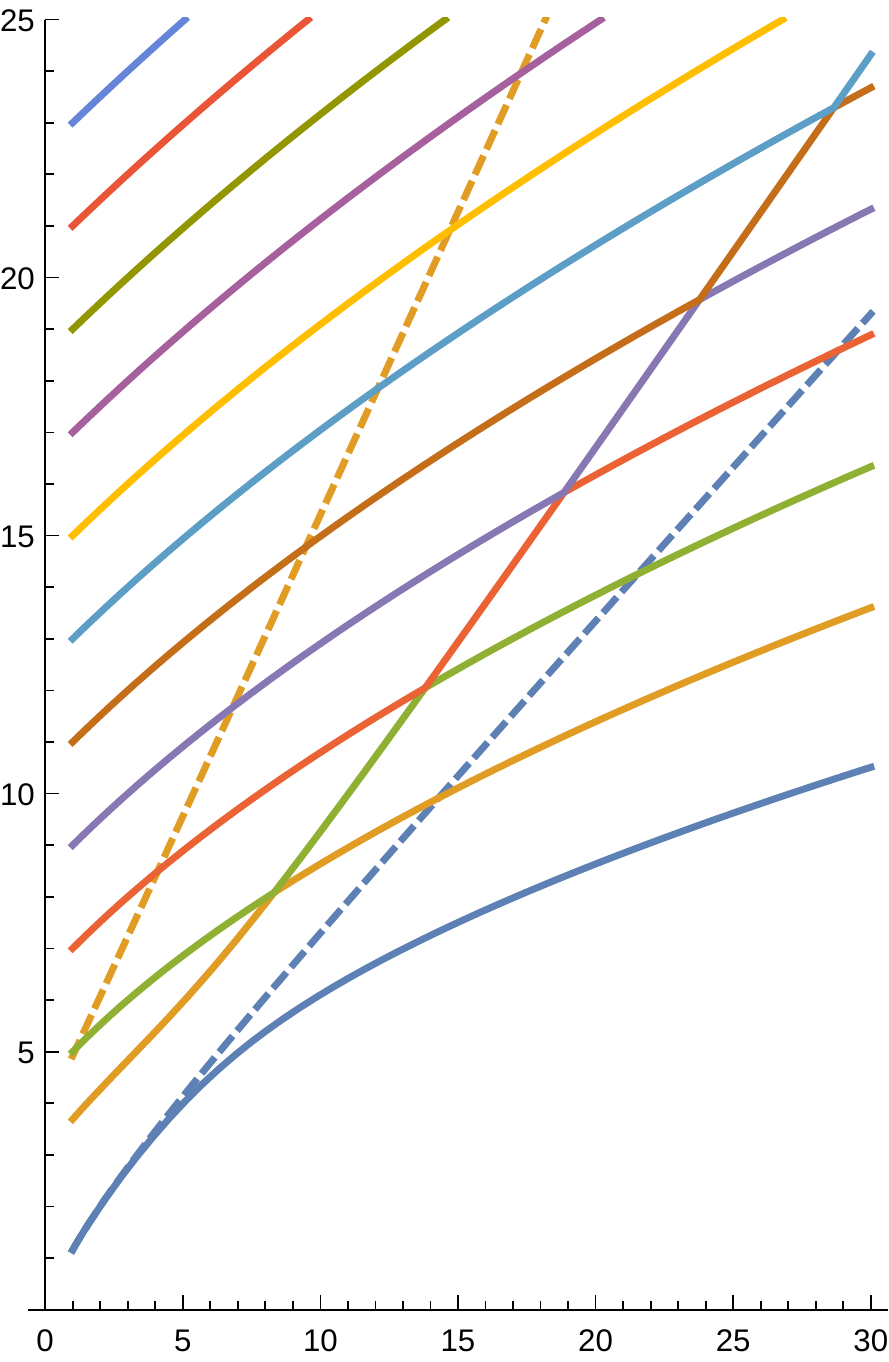} &
\includegraphics[width=0.3\textwidth]{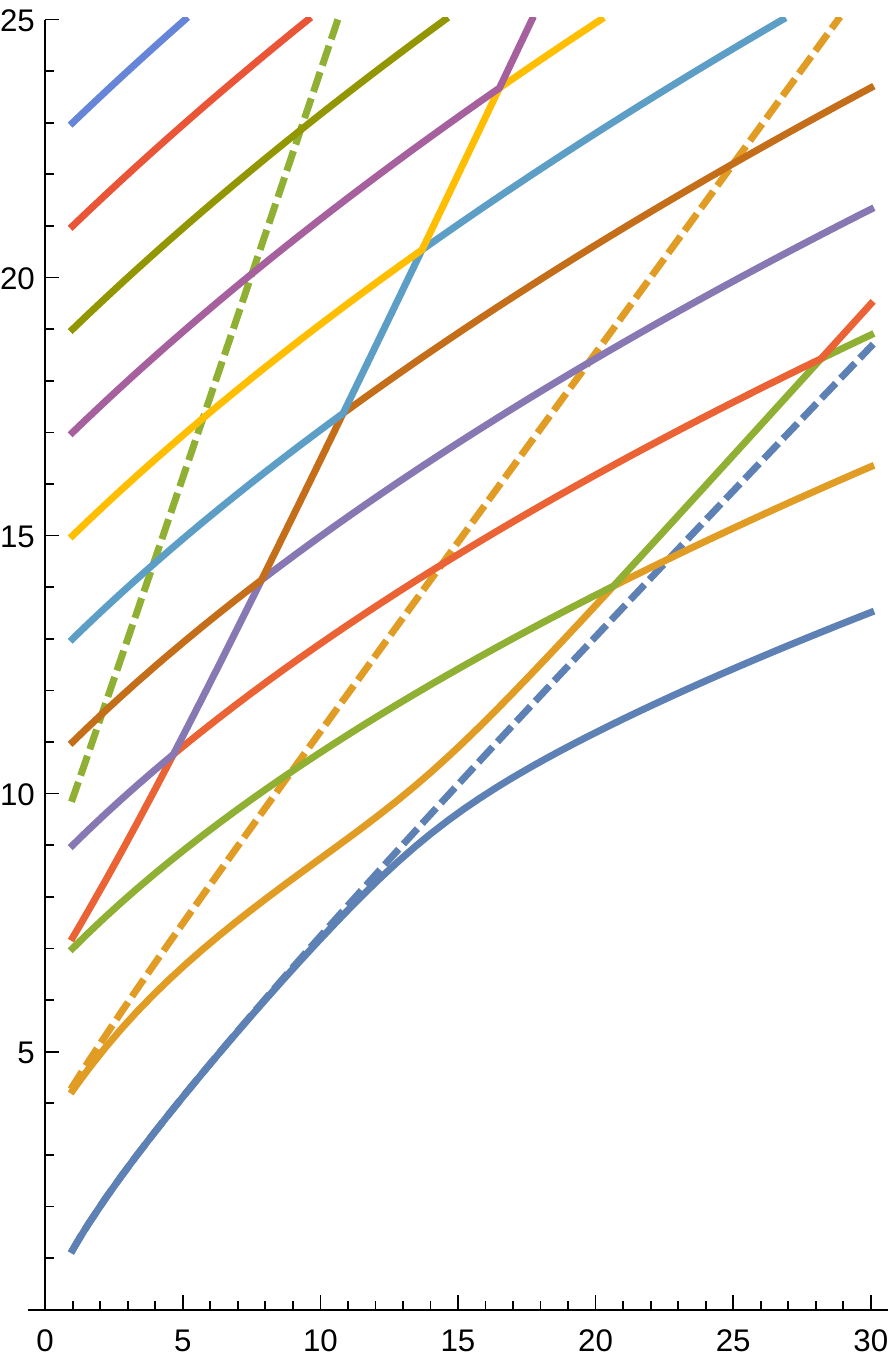}
\\
{\footnotesize $N = 1$} &
{\footnotesize $N = 2$} &
{\footnotesize $N = 3$} \\[0.5em]
\includegraphics[width=0.3\textwidth]{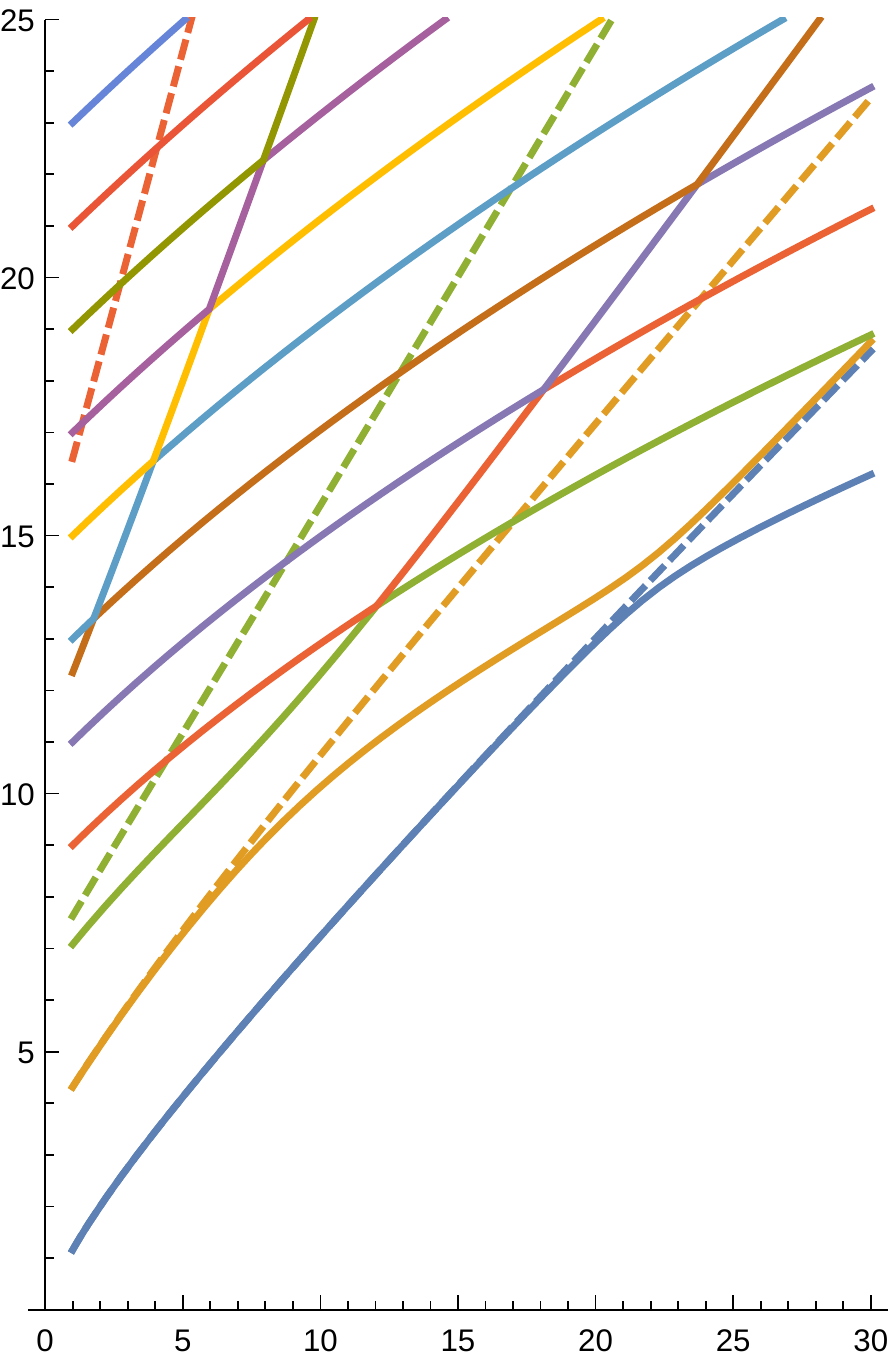} &
\includegraphics[width=0.3\textwidth]{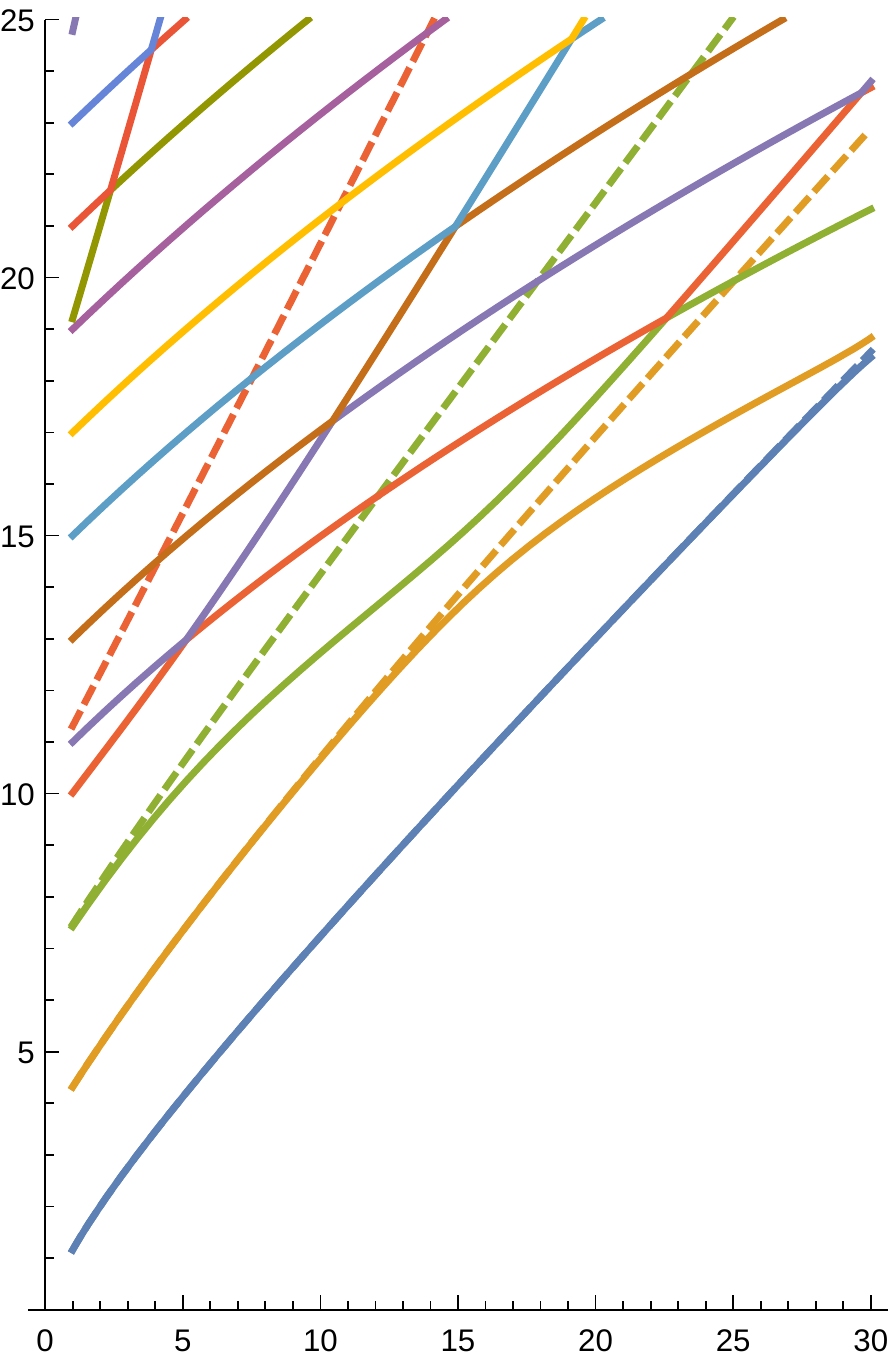} &
\includegraphics[width=0.3\textwidth]{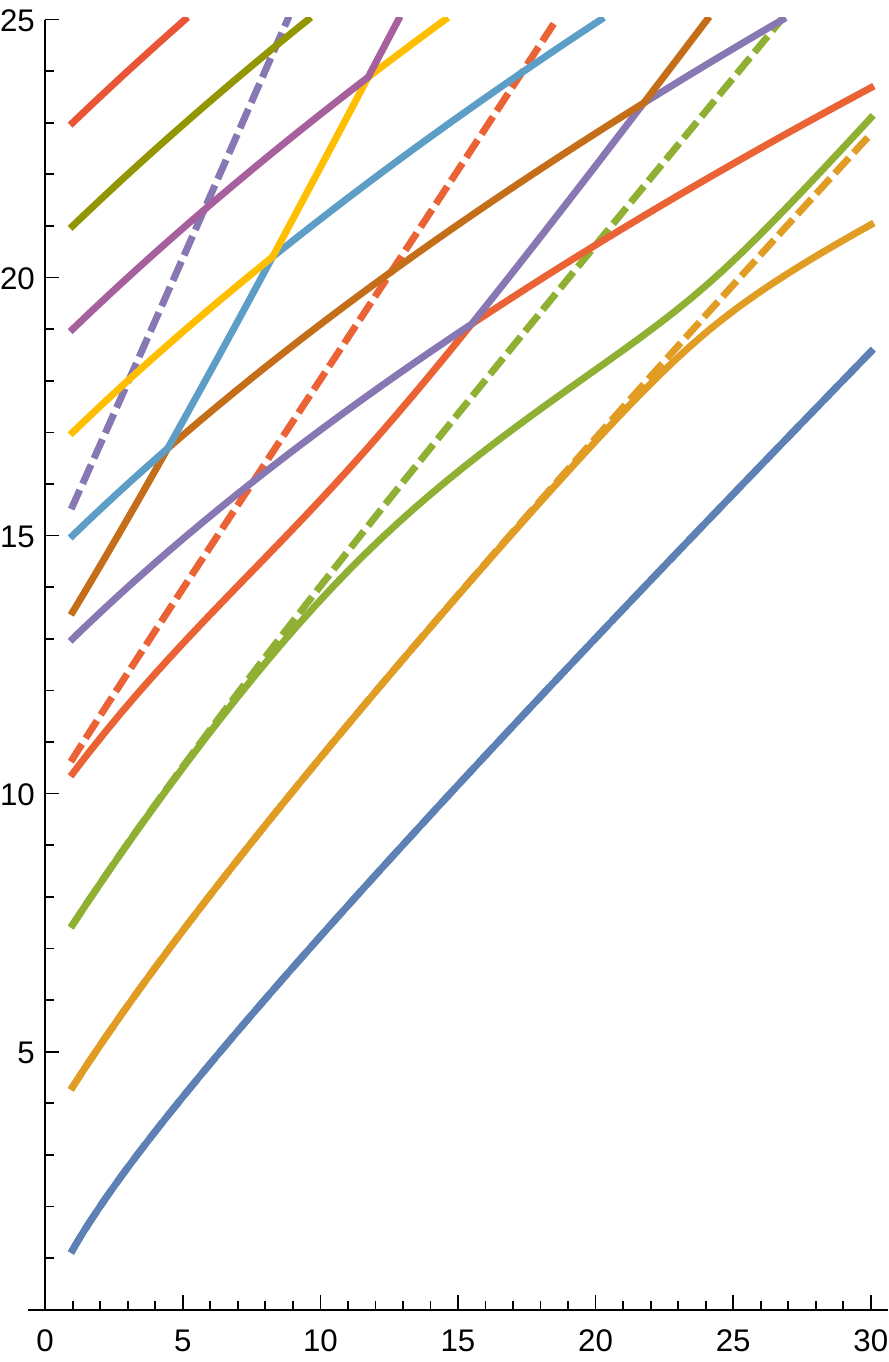}
\\
{\footnotesize $N = 4$} &
{\footnotesize $N = 5$} &
{\footnotesize $N = 6$}
\end{tabular}
\caption{Plots of lower bounds $%%\lambdal_n(d, N-1)
\mk{\lambdal_{d,n}^{(N-1)}}
$ (continuous line) and the upper bounds $\mk{\lambdau_{d,n}^{(N)}}$ (dashed line) for a continuous parameter $d \in [1, 20]$, with $\alpha = 1$ and $n = 1, 2, 3, 4, \dots$ (blue, yellow, green, red, \dots).}
\label{fig:eigenvaluesdim}
\end{figure}

\begin{table}[h!]
\centering
{\small\def\arraystretch{1.2}
\begin{tabular}{c|llllll}
& \multicolumn{1}{c}{$\alpha = 0.01$} & \multicolumn{1}{c}{$\alpha = 0.1$} & \multicolumn{1}{c}{$\alpha = 0.5$} & \multicolumn{1}{c}{$\alpha = 1$} & \multicolumn{1}{c}{$\alpha = 1.5$} & \multicolumn{1}{c}{$\alpha = 2$} \\
\hline
$N = 1$ & ${\raisebox{0.2em}{${}_{0.99}$}}^{7355322}_{4325851}$ & ${\raisebox{0.2em}{${}_{0.9}$}}^{78593744}_{51350769}$ & ${\raisebox{0.2em}{${}_{0.}$}}^{986225040}_{886226925}$ & ${\raisebox{0.2em}{${}_{1.}$}}^{178097246}_{000000000}$ & ${\raisebox{0.2em}{${}_{1.}$}}^{622347777}_{329340388}$ & ${\raisebox{0.2em}{${}_{2.}$}}^{500000000}_{000000000}$\\ 
$N = 2$ & ${\raisebox{0.2em}{${}_{0.996}$}}^{653642}_{589772}$ & ${\raisebox{0.2em}{${}_{0.972}$}}^{724341}_{237543}$ & ${\raisebox{0.2em}{${}_{0.9}$}}^{70288041}_{69571203}$ & ${\raisebox{0.2em}{${}_{1.157}$}}^{795514}_{615128}$ & ${\raisebox{0.2em}{${}_{1.597}$}}^{505864}_{338927}$ & ${\raisebox{0.2em}{${}_{2.46}$}}^{7437406}_{5598857}$\\ 
$N = 3$ & ${\raisebox{0.2em}{${}_{0.9966}$}}^{40784}_{21082}$ & ${\raisebox{0.2em}{${}_{0.972}$}}^{634579}_{486169}$ & ${\raisebox{0.2em}{${}_{0.9}$}}^{70196850}_{69976158}$ & ${\raisebox{0.2em}{${}_{1.157}$}}^{780210}_{685694}$ & ${\raisebox{0.2em}{${}_{1.597}$}}^{504264}_{482565}$ & ${\raisebox{0.2em}{${}_{2.46740}$}}^{1109}_{0328}$\\ 
$N = 4$ & ${\raisebox{0.2em}{${}_{0.9966}$}}^{37325}_{29236}$ & ${\raisebox{0.2em}{${}_{0.972}$}}^{611103}_{552418}$ & ${\raisebox{0.2em}{${}_{0.9701}$}}^{75211}_{03988}$ & ${\raisebox{0.2em}{${}_{1.1577}$}}^{75153}_{53127}$ & ${\raisebox{0.2em}{${}_{1.59750}$}}^{3629}_{0476}$ & ${\raisebox{0.2em}{${}_{2.46740110}$}}_0^1$\\ 
$N = 5$ & ${\raisebox{0.2em}{${}_{0.99663}$}}^{6075}_{1926}$ & ${\raisebox{0.2em}{${}_{0.972}$}}^{602985}_{573828}$ & ${\raisebox{0.2em}{${}_{0.9701}$}}^{69377}_{39630}$ & ${\raisebox{0.2em}{${}_{1.1577}$}}^{74241}_{67198}$ & ${\raisebox{0.2em}{${}_{1.59750}$}}^{3562}_{2809}$ & ${\raisebox{0.2em}{${}_{2.46740110}$}}_0^1$\\ 
$N = 6$ & ${\raisebox{0.2em}{${}_{0.99663}$}}^{5507}_{3068}$ & ${\raisebox{0.2em}{${}_{0.9725}$}}^{99430}_{82753}$ & ${\raisebox{0.2em}{${}_{0.9701}$}}^{67308}_{52678}$ & ${\raisebox{0.2em}{${}_{1.15777}$}}^{4009}_{1265}$ & ${\raisebox{0.2em}{${}_{1.597503}$}}^{550}_{322}$ & ${\raisebox{0.2em}{${}_{2.46740110}$}}_0^1$\\ 
$N = 7$ & ${\raisebox{0.2em}{${}_{0.99663}$}}^{5208}_{3638}$ & ${\raisebox{0.2em}{${}_{0.9725}$}}^{97620}_{87131}$ & ${\raisebox{0.2em}{${}_{0.9701}$}}^{66430}_{58373}$ & ${\raisebox{0.2em}{${}_{1.15777}$}}^{3935}_{2704}$ & ${\raisebox{0.2em}{${}_{1.597503}$}}^{547}_{464}$ & ${\raisebox{0.2em}{${}_{2.46740110}$}}_0^1$\\ 
$N = 8$ & ${\raisebox{0.2em}{${}_{0.99663}$}}^{5035}_{3955}$ & ${\raisebox{0.2em}{${}_{0.9725}$}}^{96597}_{89535}$ & ${\raisebox{0.2em}{${}_{0.97016}$}}^{6007}_{1187}$ & ${\raisebox{0.2em}{${}_{1.157773}$}}^{907}_{293}$ & ${\raisebox{0.2em}{${}_{1.5975035}$}}^{47}_{12}$ & ${\raisebox{0.2em}{${}_{2.46740110}$}}_0^1$
\end{tabular}
}
\caption{Estimates of $\lambda_{d,0}$, namely $%%\lambdau_0(d, N)
\mk{\lambdau_{d,0}^{(N)}}$ and $%%\lambdal_0(d, N-1)
\mk{\lambdal_{d,0}^{(N-1)}}
$, for $d = 1$. Here and below upper bounds are given in superscript, and lower bounds in subscript}
\label{tab:eigenvalues0}
\end{table}

\begin{table}[ht]
\centering
{\small\def\arraystretch{1.2}
\begin{tabular}{c|llllll}
& \multicolumn{1}{c}{$\alpha = 0.01$} & \multicolumn{1}{c}{$\alpha = 0.1$} & \multicolumn{1}{c}{$\alpha = 0.5$} & \multicolumn{1}{c}{$\alpha = 1$} & \multicolumn{1}{c}{$\alpha = 1.5$} & \multicolumn{1}{c}{$\alpha = 2$} \\
\hline
$N = 1$ & ${\raisebox{0.2em}{${}_{1.00}$}}^{6182160}_{1201059}$ & ${\raisebox{0.2em}{${}_{1.0}$}}^{64099168}_{15731023}$ & ${\raisebox{0.2em}{${}_{1.}$}}^{394242803}_{161869002}$ & ${}^{2.094395103}_{1.570796326}$ & ${}^{3.413006154}_{2.389104307}$ & ${}^{6.000000000}_{4.000000000}$\\ 
$N = 2$ & ${\raisebox{0.2em}{${}_{1.0047}$}}^{64277}_{44736}$ & ${\raisebox{0.2em}{${}_{1.05}$}}^{1022082}_{0866146}$ & ${\raisebox{0.2em}{${}_{1.343}$}}^{796045}_{364229}$ & ${\raisebox{0.2em}{${}_{2.00}$}}^{6175892}_{4250994}$ & ${\raisebox{0.2em}{${}_{3.2}$}}^{76235625}_{67594376}$ & ${\raisebox{0.2em}{${}_{5.7}$}}^{84128281}_{53788748}$\\ 
$N = 3$ & ${\raisebox{0.2em}{${}_{1.0047}$}}^{60636}_{44767}$ & ${\raisebox{0.2em}{${}_{1.050}$}}^{997459}_{866330}$ & ${\raisebox{0.2em}{${}_{1.343}$}}^{788026}_{505712}$ & ${\raisebox{0.2em}{${}_{2.00}$}}^{6139481}_{5947284}$ & ${\raisebox{0.2em}{${}_{3.275}$}}^{942321}_{848669}$ & ${\raisebox{0.2em}{${}_{5.7831}$}}^{86924}_{37138}$\\ 
$N = 4$ & ${\raisebox{0.2em}{${}_{1.00475}$}}^{7234}_{0212}$ & ${\raisebox{0.2em}{${}_{1.0509}$}}^{71358}_{15633}$ & ${\raisebox{0.2em}{${}_{1.343}$}}^{748669}_{653052}$ & ${\raisebox{0.2em}{${}_{2.006}$}}^{123190}_{081438}$ & ${\raisebox{0.2em}{${}_{3.27593}$}}^{7929}_{0257}$ & ${\raisebox{0.2em}{${}_{5.7831859}$}}^{64}_{37}$\\ 
$N = 5$ & ${\raisebox{0.2em}{${}_{1.00475}$}}^{5934}_{2117}$ & ${\raisebox{0.2em}{${}_{1.0509}$}}^{61820}_{32405}$ & ${\raisebox{0.2em}{${}_{1.343}$}}^{737534}_{694385}$ & ${\raisebox{0.2em}{${}_{2.0061}$}}^{20326}_{05496}$ & ${\raisebox{0.2em}{${}_{3.27593}$}}^{7624}_{5484}$ & ${\raisebox{0.2em}{${}_{5.78318596}$}_2^3}$\\ 
$N = 6$ & ${\raisebox{0.2em}{${}_{1.00475}$}}^{5316}_{2991}$ & ${\raisebox{0.2em}{${}_{1.0509}$}}^{57437}_{39971}$ & ${\raisebox{0.2em}{${}_{1.3437}$}}^{33293}_{11001}$ & ${\raisebox{0.2em}{${}_{2.00611}$}}^{9514}_{3354}$ & ${\raisebox{0.2em}{${}_{3.27593}$}}^{7560}_{6864}$ & ${\raisebox{0.2em}{${}_{5.78318596}$}_2^3}$\\ 
$N = 7$ & ${\raisebox{0.2em}{${}_{1.00475}$}}^{4982}_{3449}$ & ${\raisebox{0.2em}{${}_{1.0509}$}}^{55141}_{43873}$ & ${\raisebox{0.2em}{${}_{1.3437}$}}^{31413}_{18702}$ & ${\raisebox{0.2em}{${}_{2.00611}$}}^{9238}_{6345}$ & ${\raisebox{0.2em}{${}_{3.275937}$}}^{544}_{279}$ & ${\raisebox{0.2em}{${}_{5.78318596}$}_2^3}$\\ 
$N = 8$ & ${\raisebox{0.2em}{${}_{1.00475}$}}^{4785}_{3713}$ & ${\raisebox{0.2em}{${}_{1.0509}$}}^{53819}_{46092}$ & ${\raisebox{0.2em}{${}_{1.3437}$}}^{30479}_{22677}$ & ${\raisebox{0.2em}{${}_{2.00611}$}}^{9130}_{7637}$ & ${\raisebox{0.2em}{${}_{3.275937}$}}^{539}_{426}$ & ${\raisebox{0.2em}{${}_{5.78318596}$}_2^3}$
\end{tabular}
}
\caption{Estimates of $\lambda_{d,0}$, namely $%%\lambdau_0(d, N)
\mk{\lambdau_{d,0}^{(N)}}
$ and $%%\lambdal_0(d, N-1)
\mk{\lambdal_{d,0}^{(N-1)}}
$, for $d = 2$}
\label{tab:eigenvalues1}
\end{table}

\begin{table}[ht]
\centering
{\small\def\arraystretch{1.2}
\begin{tabular}{c|llllll}
& \multicolumn{1}{c}{$\alpha = 0.01$} & \multicolumn{1}{c}{$\alpha = 0.1$} & \multicolumn{1}{c}{$\alpha = 0.5$} & \multicolumn{1}{c}{$\alpha = 1$} & \multicolumn{1}{c}{$\alpha = 1.5$} & \multicolumn{1}{c}{$\alpha = 2$} \\
\hline
$N = 1$ & ${\raisebox{0.2em}{${}_{1.0}$}}^{22142633}_{11074238}$ & ${\raisebox{0.2em}{${}_{1.}$}}^{241761168}_{118753169}$ & ${}^{2.847446732}_{1.827843033}$ & ${}^{7.775441818}_{3.657142857}$ & ${}^{21.189848505}_{\phantom{0}7.869220333}$ & ${}^{58.500000000}_{18.000000000}$\\ 
$N = 2$ & ${\raisebox{0.2em}{${}_{1.01}$}}^{8704583}_{6785647}$ & ${\raisebox{0.2em}{${}_{1.}$}}^{204060727}_{182127557}$ & ${\raisebox{0.2em}{${}_{2.}$}}^{553470124}_{350636474}$ & ${}^{6.677830073}_{5.762545969}$ & ${\raisebox{0.2em}{${}_{1}$}}^{7.911477317}_{4.668543888}$ & ${}^{49.295799157}_{38.643391195}$\\ 
$N = 3$ & ${\raisebox{0.2em}{${}_{1.0185}$}}^{98071}_{19922}$ & ${\raisebox{0.2em}{${}_{1.20}$}}^{2849999}_{1981651}$ & ${\raisebox{0.2em}{${}_{2.5}$}}^{42649925}_{35387077}$ & ${\raisebox{0.2em}{${}_{6.6}$}}^{31881932}_{01605246}$ & ${\raisebox{0.2em}{${}_{17.}$}}^{760546718}_{657335339}$ & ${\raisebox{0.2em}{${}_{48.}$}}^{841213791}_{503783872}$\\ 
$N = 4$ & ${\raisebox{0.2em}{${}_{1.01858}$}}^{8891}_{6584}$ & ${\raisebox{0.2em}{${}_{1.2027}$}}^{55722}_{35425}$ & ${\raisebox{0.2em}{${}_{2.5420}$}}^{90312}_{32407}$ & ${\raisebox{0.2em}{${}_{6.630}$}}^{296958}_{170607}$ & ${\raisebox{0.2em}{${}_{17.75}$}}^{6624774}_{5998202}$ & ${\raisebox{0.2em}{${}_{48.8}$}}^{31283297}_{28376987}$\\ 
$N = 5$ & ${\raisebox{0.2em}{${}_{1.01858}$}}^{8512}_{6588}$ & ${\raisebox{0.2em}{${}_{1.2027}$}}^{52564}_{35527}$ & ${\raisebox{0.2em}{${}_{2.5420}$}}^{85179}_{38329}$ & ${\raisebox{0.2em}{${}_{6.6302}$}}^{96943}_{58528}$ & ${\raisebox{0.2em}{${}_{17.756}$}}^{610435}_{589783}$ & ${\raisebox{0.2em}{${}_{48.8311}$}}^{94039}_{79729}$\\ 
$N = 6$ & ${\raisebox{0.2em}{${}_{1.01858}$}}^{8127}_{6814}$ & ${\raisebox{0.2em}{${}_{1.2027}$}}^{49211}_{37864}$ & ${\raisebox{0.2em}{${}_{2.5420}$}}^{76948}_{49917}$ & ${\raisebox{0.2em}{${}_{6.6302}$}}^{91620}_{75546}$ & ${\raisebox{0.2em}{${}_{17.75660}$}}^{8599}_{4982}$ & ${\raisebox{0.2em}{${}_{48.8311936}$}}^{45}_{03}$\\ 
$N = 7$ & ${\raisebox{0.2em}{${}_{1.01858}$}}^{7899}_{6939}$ & ${\raisebox{0.2em}{${}_{1.2027}$}}^{47286}_{39124}$ & ${\raisebox{0.2em}{${}_{2.5420}$}}^{72930}_{55189}$ & ${\raisebox{0.2em}{${}_{6.63028}$}}^{9740}_{0585}$ & ${\raisebox{0.2em}{${}_{17.75660}$}}^{8285}_{6462}$ & ${\raisebox{0.2em}{${}_{48.83119364}$}_3^4}$\\ 
$N = 8$ & ${\raisebox{0.2em}{${}_{1.018587}$}}^{751}_{025}$ & ${\raisebox{0.2em}{${}_{1.2027}$}}^{46054}_{39984}$ & ${\raisebox{0.2em}{${}_{2.5420}$}}^{70634}_{58531}$ & ${\raisebox{0.2em}{${}_{6.63028}$}}^{8837}_{3411}$ & ${\raisebox{0.2em}{${}_{17.75660}$}}^{8157}_{7242}$ & ${\raisebox{0.2em}{${}_{48.83119364}$}_3^4}$\end{tabular}
}
\caption{Estimates of $\lambda_{d,0}$, namely $%%\lambdau_0(d, N)
\mk{\lambdau_{d,0}^{(N)}}
$ and $%%\lambdal_0(d, N-1)
\mk{\lambdal_{d,0}^{(N-1)}}
$, for $d = 9$.}
\label{tab:eigenvalues3}
\end{table}

\begin{table}[ht]
\centering
{\small\def\arraystretch{1.2}
\begin{tabular}{c|llllll}
& \multicolumn{1}{c}{$\alpha = 0.01$} & \multicolumn{1}{c}{$\alpha = 0.1$} & \multicolumn{1}{c}{$\alpha = 0.5$} & \multicolumn{1}{c}{$\alpha = 1$} & \multicolumn{1}{c}{$\alpha = 1.5$} & \multicolumn{1}{c}{$\alpha = 2$} \\
\hline
$N = 1$ & ${}^{\infty}_{1.019691102}$ & ${}^{\infty}_{1.216080194}$ & ${}^{\infty}_{2.696537403}$ & ${}^{\infty}_{7.516505860}$ & ${}^{\infty}_{21.614073847}$ & ${}^{\infty}_{\phantom{0}64.000000000}$\\ 
$N = 2$ & ${}^{\infty}_{1.019691102}$ & ${}^{\infty}_{1.216080194}$ & ${}^{\infty}_{2.696537403}$ & ${}^{\infty}_{7.516505860}$ & ${}^{\infty}_{21.614073847}$ & ${}^{\infty}_{\phantom{0}64.000000000}$\\ 
$N = 3$ & ${}^{\infty}_{1.021053433}$ & ${}^{\infty}_{1.231703408}$ & ${}^{\infty}_{2.840043505}$ & ${}^{\infty}_{8.131382095}$ & ${}^{\infty}_{23.590433751}$ & ${}^{\infty}_{\phantom{0}69.593538821}$\\ 
$N = 4$ & ${\raisebox{0.2em}{${}_{1.0}$}}^{32479732}_{22241923}$ & ${\raisebox{0.2em}{${}_{1.}$}}^{371086873}_{246672211}$ & ${}^{4.530615391}_{3.044137927}$ & ${}^{18.452353592}_{\phantom{0}9.513077729}$ & ${}^{71.419348109}_{30.479221324}$ & ${}^{269.510897721}_{100.000000000}$\\ 
$N = 5$ & ${\raisebox{0.2em}{${}_{1.02}$}}^{5380240}_{3629216}$ & ${\raisebox{0.2em}{${}_{1.2}$}}^{84853223}_{63489904}$ & ${\raisebox{0.2em}{${}_{3.}$}}^{504465514}_{243864464}$ & ${\raisebox{0.2em}{${}_{1}$}}^{2.333689344}_{0.716443875}$ & ${}^{43.715868468}_{36.043117616}$ & ${\raisebox{0.2em}{${}_{1}$}}^{56.360051134}_{23.377854920}$\\ 
$N = 6$ & ${\raisebox{0.2em}{${}_{1.02}$}}^{4439711}_{3944192}$ & ${\raisebox{0.2em}{${}_{1.2}$}}^{73412947}_{67411652}$ & ${\raisebox{0.2em}{${}_{3.}$}}^{367493681}_{295975897}$ & ${\raisebox{0.2em}{${}_{11.}$}}^{510711118}_{072719899}$ & ${\raisebox{0.2em}{${}_{3}$}}^{9.961582676}_{7.886428493}$ & ${\raisebox{0.2em}{${}_{1}$}}^{40.917773984}_{31.938621312}$\\ 
$N = 7$ & ${\raisebox{0.2em}{${}_{1.0242}$}}^{83730}_{46278}$ & ${\raisebox{0.2em}{${}_{1.271}$}}^{573161}_{116171}$ & ${\raisebox{0.2em}{${}_{3.34}$}}^{7850527}_{2181181}$ & ${\raisebox{0.2em}{${}_{11.}$}}^{404208370}_{367799234}$ & ${\raisebox{0.2em}{${}_{39.}$}}^{508972239}_{330081434}$ & ${\raisebox{0.2em}{${}_{13}$}}^{9.145260868}_{8.349747040}$\\ 
$N = 8$ & ${\raisebox{0.2em}{${}_{1.02427}$}}^{7661}_{4471}$ & ${\raisebox{0.2em}{${}_{1.2714}$}}^{97834}_{63292}$ & ${\raisebox{0.2em}{${}_{3.346}$}}^{898372}_{612135}$ & ${\raisebox{0.2em}{${}_{11.39}$}}^{8420985}_{7008167}$ & ${\raisebox{0.2em}{${}_{39.4}$}}^{83166347}_{76941411}$ & ${\raisebox{0.2em}{${}_{139.0}$}}^{43142252}_{16247816}$
\end{tabular}
}
\caption{Estimates of $\lambda_{d,3}$, namely $%%\lambdau_3(d, N)
\mk{\lambdau_{d,3}^{(N)}}
$ and $%%\lambdal_3(d, N-1)
\mk{\lambdal_{d,3}^{(N-1)}}
$, for $d = 2$.}
\label{tab:eigenvalues2}
\end{table}

%
%                            ---------- o ----------
%

\end{document}